\documentclass{article}
\usepackage{amsfonts}
\usepackage{amsmath}
\usepackage{amssymb}
\usepackage{amsthm}
\usepackage{xcolor} 
\usepackage{mathrsfs}
\usepackage{bm}
\usepackage{enumerate}
\usepackage{hyperref}
\usepackage{authblk}
\usepackage[all]{xy}
\newcommand{\ud}{\mathrm{d}}
\newcommand{\dev}{\partial}
\newcommand{\R}{\mathbb{R}}
\newcommand{\Z}{\mathbb{Z}}
\newcommand{\cA}{\mathcal{A}}
\newcommand{\M}{\mathcal{M}}

\newcommand{\cF}{\mathcal{F}}

\newcommand{\mdev}{\boldsymbol{\dev}}
\newcommand{\mlambda}{\boldsymbol{\lambda}}
\newcommand{\mmu}{\boldsymbol{\mu}}

\renewcommand{\L}{\mathcal{L}}

\renewcommand{\vec}[1]{\mathbf{#1}}
\newtheorem{thm}{Theorem}

\newtheorem{cor}[thm]{Corollary}
\theoremstyle{definition}
\newtheorem{deff}{Definition}
\newtheorem{rmk}{Remark}
\newtheorem{eg}{Example}
\DeclareMathOperator{\Der}{Der}
\DeclareMathOperator{\im}{Im}
\hypersetup{colorlinks=true}

\title{Higher order dispersive deformations of multidimensional Poisson brackets of hydrodynamic type}
\author{Matteo Casati}
\affil{Marie Curie fellow of the Istituto Nazionale di Alta Matematica\\Department of Mathematical Sciences, Loughborough University\\Loughborough (United Kingdom) LE11 3TU}

\begin{document}
\maketitle
\begin{abstract}
The theory of multidimensional Poisson vertex algebras (mPVAs) provides a completely algebraic formalism to study the Hamiltonian structure of PDEs, for any number of dependent and independent variables. In this paper, we compute the cohomology of the PVAs associated with two-dimensional, two-components Poisson brackets of hydrodynamic type at the third differential degree. This allows us to obtain their corresponding Poisson--Lichnerowicz cohomology, which is the main building block of the theory of their deformations. Such a cohomology is trivial neither in the second group, corresponding to the existence of a class of not equivalent infinitesimal deformation, nor in the third, corresponding to the obstructions to extend such deformations.
\end{abstract}
\paragraph*{Keywords}  Hamiltonian operators, Hydrodynamic Poisson brackets, Poisson Vertex Algebras
\paragraph*{MSC} 37K05 (primary), 37K25, 17B80
\section{Introduction}\label{sec:intro}

In a previous paper of ours \cite{MC15-1}, we introduced the notion of multidimensional Poisson Vertex Algebras (mPVAs) and demonstrated how they prove themselves a powerful tool for dealing with Poisson structures on a space of functions depending on more than one (space) variables. Such a general class of brackets, whose elements are the Hamiltonian structures for evolutionary PDEs, is of paramount importance for the study of the integrability properties of the equations. It is possible, for instance, to prove the integrability for many of such PDEs provided that they are Hamiltonian with respect to two compatible Poisson structures \cite{M78}.

Among all the Poisson brackets defined on a given space of functions, that formally is identified by a fixed number $n$ of components and $d$ space variables, the most studies class is the one of the so-called Poisson brackets \emph{of hydrodynamic type}, namely the ones that are defined in terms of homogeneous differential operators of first order. In particular, there exists a vast literature about Poisson brackets of hydrodynamic type for $d=1$ \cite{DN83, G01, dGMS05, DZ}, since they provide Hamiltonian structures for several integrable PDEs (Korteweg-de Vries \cite{g71}, Benney moment chain \cite{B73}, Harry Dym \cite{M80} to name a few). Their general form is \cite{DN83}
\begin{equation}\label{eq:htpb-intro}
\{u^i(x),u^j(y)\}=\left(g^{ij}(u(x))\frac{\ud}{\ud x}+b^{ij}_ku^k_x\right)\delta(x-y)
\end{equation}
for $i,j=1,\ldots,n$. The properties that $g$ and $b$ must satisfy for the bracket defined in \eqref{eq:htpb-intro} to be nondegenerate and Hamiltonian guarantee the existence of a change of coordinates of the form $v^i(x)=f^i(u(x))$ such that the bracket takes the form
\begin{equation}\label{eq:htpbflat-intro}
\{v^i(x),v^j(y)\}=\eta^{ij}\frac{\ud}{\ud x}\delta(x-y)
\end{equation}
with $\eta^{ij}$ constant.

We call a \emph{dispersive deformation} of a Poisson bracket of hydrodynamic type a bracket defined by an higher order differential operator which is compatible, up to its order, with the first order one. In general, they have the form
\begin{equation}\label{eq:gendefo-intro}
\{u^i(x),u^j(y)\}=\{u^i(x),u^j(y)\}_0+\sum_k\{u^i(x),u^j(y)\}_{[k]}
\end{equation}
with $\{u^i(x),u^j(y)\}_0$ of hydrodynamic type. It is a well known fact, proved by several authors \cite{G01,dGMS05,DZ}, that all the deformations of Poisson brackets of hydrodynamic type for $d=1$ are trivial, namely there exists a change of coordinates of the form
\begin{equation}
v^i=u^i+\sum_{k=1}^\infty F^i_{[k]}(u,u_x,\ldots,u_{kx}),
\end{equation}
called a general Miura type transformation, for which the deformed bracket takes the form \eqref{eq:htpb-intro}.

The classification of all the compatible deformations of some Poisson bracket of hydrodynamic type, up to Miura transformation, is closely connected with the computation of the Poisson-Lichnerowicz cohomology of the Poisson bracket of hydrodynamic type itself. In the $d=1$ case, the triviality result follows, indeed, from the fact that all the cohomology group $H^p$ of a bracket of form \eqref{eq:htpbflat-intro} are trivial for $p>1$, and that $H^1$ is of a very special form. The precise connection between the Poisson-Lichnerowicz cohomology and the theory of the deformations will be discussed in Section 2.1. In short, the nontrivial infinitesimal deformations (for which the terms $\{u^i(x),u^j(y)\}_{[k]}$ of formula \eqref{eq:gendefo-intro} vanish for all but one value of $k$) are parametrized by the second cohomology group, and the elements of the third group are the obstructions to the extension from an infinitesimal deformation to a full deformation, adding higher order terms.

The natural generalisation of this classification problem is the study the deformations of \emph{multidimensional} Poisson brackets of hydrodynamic type, for which $d>1$. The landscape is at the same time richer and much more complicated, since there is not a preferred coordinate system where the bracket can be brought in constant form \cite{DN84, M08}. In this paper we focus on the case $d=2$, where the Poisson brackets of hydrodynamic type have been completely classified up to $n=3$ components \cite{FLS15}. In particular we have investigated the $n=2$ component case, where there are three normal forms for the bracket of hydrodynamic type.
 
There, we computed the first and second groups of the Poisson cohomology: the results reported in \cite{MC15-1} hinted that the picture might not be the same in the two-dimensional case: at the first dispersive order we found that the first cohomology group of two-dimensional, two-components Poisson brackets of hydrodynamic type is in general not trivial. Nevertheless, the component of the second cohomology group corresponding to the infinitesimal deformations of first order ($k=1$ of \eqref{eq:gendefo-intro}, or brackets homogeneous of differential degree 2) is trivial.

First and second order deformations of a scalar reduction of a two-component bracket of hydrodynamic type are trivial, too \cite{MC16}.

However, a deeper investigation shows that, in general, the two-dimensional picture is much more complicated than the one-dimensional case. Already for scalar, i.e. one-component, brackets the dimensions of the cohomology groups grow as the degree of the components grows, ultimately making the Poisson cohomology infinite dimensional \cite{CCS15}.  

In this paper, we present the results for the computation of the second degree component of the second Poisson cohomology group for $(d=2,n=2)$ brackets of hydrodynamic type. Such group is never trivial, hence there exist infinitesimal deformations of the brackets that are not trivial. However, as we found for the deformations in the scalar case \cite{CCS17}, very few of these infinitesimal deformations can be extended to a finite one. 

\subsection{Brackets of hydrodynamic type and Poisson Vertex Algebras}
The notion of Poisson brackets of hydrodynamic type has been introduced in the early 1980s by Dubrovin and Novikov \cite{DN83} as a unifying picture for a large class of integrable Hamiltonian PDEs, and more in general to describe in field-theoretical terms Hamiltonian systems with internal degrees of freedom.

In this paragraph we recall their definition, both in the one- and in the multi-dimensional case, and we briefly summarise the discussion already presented in \cite{MC15-1} explaining why and how multidimensional Poisson Vertex Algebras are a convenient equivalent language to deal with the problem addressed in this paper.

Let $M$ be a compact, $n$-dimensional real manifold. Let us consider $L(M)=C^\infty(S^1,M)$ the \emph{loop space} of $M$, namely the space of periodic functions of one variable with values in the manifold. The elements of $L(M)$ are $n$-tuples of periodic functions $\{u^i(x)\}_{i=1}^n$, where $\{u^i\}$ are local coordinates on the manifold $M$. A (one-dimensional) Poisson bracket of hydrodynamic type on $L(M)$ is an homogeneous first order differential operator acting on the Dirac's delta
\begin{equation}\label{eq:HYPB-def1}
\{u^i(x),u^j(y)\}=\left(g^{ij}(u(x))\frac{\ud}{\ud x}+b^{ij}_k(u(x))u^k_x\right)\delta(x-y)
\end{equation}
that endows the space $\cF$ of local functionals on the loop space with the structure of a Lie algebra.

Given local functionals of the form
\begin{equation}\label{eq:locfunc-1}
F=\int_{S^1} f(u(x),\dev_x u(x),\dev^2_x u(x),\ldots,\dev^k_xu(x))\ud x,
\end{equation}
the bracket on $\cF$ is defined by
\begin{equation}\label{eq:LieB-def1}
\{F,G\}=\iint \frac{\delta F}{\delta u^i(x)}\{u^i(x),u^j(y)\}\frac{\delta G}{\delta u^j(y)}\ud x\ud y,
\end{equation}
which in the case of a bracket of hydrodynamic type reads
\begin{equation}\label{eq:LieB-ht1}
\{F,G\}=\int\frac{\delta F}{\delta u^i}\left(g^{ij}\frac{\ud}{\ud x}\frac{\delta G}{\delta u^j}+b^{ij}_ku^k_x\frac{\delta G}{\delta u^j}\right)\ud x.
\end{equation}
We use a formal version of the space of local functionals, as introduced by Gel'fand and {Diki\u\i} \cite{GD75}. Let us consider the coordinates $\{u^i\}$ on $M$ and their corresponding jet variables $\{u^i_{kx}\}$, where $\dev^k_xu^i=u^i_{kx}$. We allow the \emph{densities} $f$ of the local functionals to be differential polynomials, namely elements of the space
$$
\cA:=C^\infty(u^i)[u^i_{kx}]\qquad\qquad i=1,\ldots,n,\; k\geq 1.
$$ 
We regard the integration as a projection map between the space of differential polynomials and the space of local functionals, whose kernel are total derivatives and where there are never boundary terms because of the assumed periodicity of the functions. Hence,
$$
\cF=\frac{\cA}{\dev_x\cA},
$$
and the variational derivative is formally defined by the Euler-Lagrange formula
\begin{equation}
\frac{\delta F}{\delta u^i}=\sum_{k\geq 0}\left(-\frac{\ud}{\ud x}\right)^k\frac{\dev f}{\dev u^i_{kx}}.
\end{equation}

The requirement that the bracket at the LHS of \eqref{eq:LieB-ht1} is skewsymmetric and fulfils Jacobi's identity dictates a set of conditions for the functions $g^{ij}(u)$ and $b^{ij}_k(u)$. As first proved by Dubrovin and Novikov \cite{DN83}, in the nondegenerate case $g^{ij}$ must be a flat contravariant pseudo-Riemannian metric on $M$ and $b^{ij}_k$ the contravariant Christoffel symbols of its Levi-Civita connection.

The generalisation of this picture to the multidimensional case is straightforward, but it already introduces new constraints that make it more technically cumbersome. We consider the space $\M$ of the maps between a $d$-dimensional compact manifold $\Sigma$ and $M$. In a local system of coordinates, they are $\{u^i(x^1,\ldots,x^d)\}_{i=1}^n$. To avoid topological problems when working with the local functionals, let us assume that $\Sigma\cong T^d$, so that the integration by parts is always allowed variable-wise. The suitable space of differential polynomials we have to consider when working in a multidimensional setting requires working with jet variables with respect to all the $d$ total derivatives, for which it is convenient to adopt a multi-index notation as follows. Let $\cA$ be the space of differential polynomials. We have
$$\cA=C^\infty(\{u^i\}_{i=1}^n)[\{u^i_L, i=1\ldots n,\,L\in \Z^d_{\geq 0}\}],
$$
where $L=(l_1,\ldots,l_d)$ is an index counting the number of derivatives with respect to each independent variables, namely
$$
u^i_L=\dev_{x^1}^{l_1}\dev_{x^2}^{l_2}\cdots\dev_{x^d}^{l_d}u^i.
$$
The corresponding space of local functionals is
$$
\cF=\frac{\cA}{\sum_{\alpha=1}^d\dev_{x^\alpha}\cA}
$$
and, consequently, the variational derivative is
\begin{equation}
\frac{\delta F}{\delta u^i}=\sum_{l_1,\ldots,l_d}\left(\prod_{\alpha=1}^d\left(-\frac{\ud}{\ud x^\alpha}\right)^{l_\alpha}\right)\frac{\dev f}{\dev u^i_L}=\sum_L\left(-\mdev\right)^L\frac{\dev f}{\dev u^i_L},
\end{equation} where we also denoted the shorthand notation we will adopt where no ambiguity can arise.

The general form of a multidimensional Poisson bracket of hydrodynamic type is \cite{DN84}
\begin{equation}\label{eq:mHYPB-def}
\{u^i(\vec{x}),u^j(\vec{y})\}=\sum_{\alpha=1}^d\left(g^{\alpha ij}(u(\vec{x}))\frac{\ud}{\ud x^\alpha}+b^{\alpha ij}_ku^k_\alpha\right)\delta(x^1-y^1)\cdots\delta(x^d-y^d),
\end{equation}
where $u^i_\alpha=\dev_{x^\alpha} u^i$. If $g^\alpha$ are invertible matrices, it is defined by $d$ flat pseudo-Riemannian metrics, where $b^\alpha$'s are the contravariant Christoffel symbols of the corresponding Levi-Civita connections, subjected to additional compatibility conditions \cite{M88}. We introduce the standard Christoffel symbols $\Gamma_{\alpha ij}^k=-g^\alpha_{is}b^{\alpha sk}_j$ and their corresponding Riemannian curvature tensors $R^{\alpha l}_{ijk}$, and define the obstruction tensors $T^{\alpha\beta i}_{jk}=\Gamma^{\beta i}_{jk}-\Gamma^{\alpha i}_{jk}$ and $T^{\alpha\beta ijk}=g^{\alpha ia}g^{\beta kb}T^{\alpha\beta j}_{ab}$. Then, the additional compatibility conditions are
\begin{subequations}
\begin{gather}
T^{\alpha\beta ijk}=T^{\alpha \beta kji}\\
T^{\alpha\beta ijk}+T^{\alpha\beta jki}+T^{\alpha\beta kji}=0\\
T^{\alpha\beta ijl}T^{\alpha\beta k}_{lr}=T^{\alpha\beta ikl}T^{\alpha\beta j}_{lr}\\\label{eq:condlinear}
\nabla^\alpha_r T^{\alpha\beta ijk}=0.
\end{gather}
\end{subequations}
The main difference with the one-dimensional case is that, while in the former, because of the flatness of the matric, there exists a system of coordinates in which the bracket has constant form and the Christoffel symbols vanish, in the multidimensional case this can happen only if the obstruction tensors vanish. However, Equation \eqref{eq:condlinear} guarantees that, in the system of coordinates in which at least one of the $d$ metrics is constant, the other ones are at most linear. This fact is of crucial importance for the classification of the brackets.

Proving the results sketched in this Paragraph is mostly a matter of computation, whose complexity steadily grows with the number of dependent and independent variables taken into account. As previously claimed \cite{BdSK09, MC15-1}, the notion of Poisson Vertex Algebra is a very powerful algebraic counterpart of Poisson brackets in the space of local functionals, so we decide to briefly recall the definition of multidimensional Poisson Vertex Algebras and to cast our problem using their language.

\begin{deff}\label{def:PVA}
A $d$-dimensional Poisson Vertex Algebra (PVA) is a differential algebra $(\cA,\{\dev_\alpha\}_{\alpha=1}^d)$ endowed with  $d$ commuting derivations and with a bilinear operation $\{\cdot_{\mlambda}\cdot\}\colon \cA\otimes \cA\to \R[\lambda_1,\ldots,\lambda_d]\otimes \cA$ called the $\lambda$ \emph{bracket}, taking value in the polynomials of $d$ formal indeterminates $\lambda_\alpha$ with coefficients in $\cA$ and satisfying the following set of properties:
\begin{enumerate}
\item $\{\dev_\alpha f_{\mlambda} g\}=-\lambda_\alpha\{f_{\mlambda} g\}$
\item $\{f_{\mlambda}\dev_\alpha g\}=\left(\dev_\alpha+\lambda_\alpha\right)\{f_{\mlambda} g\}$
\item $\{f_{\mlambda} gh\}=\{f_{\mlambda} g\}h+\{f_{\mlambda} h\}g$
\item $\{fg_{\mlambda} h\}=\{f_{\mlambda+\mdev}h\}g+\{g_{\mlambda+\mdev}h\}f$
\item $\{g_{\mlambda}f\}=-{}_{\to}\{f_{-\mlambda-\mdev}g\}$ (PVA-skewsymmetry)
\item $ \{f_{\mlambda}\{g_{\mmu}h\}\}-\{g_{\mmu}\{f_{\mlambda}h\}\}=\{\{f_{\mlambda}g\}_{\mlambda+\mmu}h\}$ (PVA-Jacobi identity).
\end{enumerate}
\end{deff}
For our purposes, we consider PVAs where the differential algebra is precisely the space of differential polynomials, and the $d$ derivations $\{\dev_\alpha\}$ correspond to the derivatives with respect to the independent variables.

To read the properties and the following computation, let us recall that the $\lambda$ bracket between two differential polynomials is a polynomial in $\lambda$'s. Using the same multi-index notation we have already adopted, let $\mlambda^L=\lambda_1^{l_1}\lambda_2^{l_2}\cdots\lambda_d^{l_d}$ and $\{f_{\mlambda}g\}=\sum_L B(f,g)_L \mlambda^L$, where $B(f,g)_L\in\cA$. Hence, the terms in the RHS of Property (4) are of the form $\{f_{\mlambda+\mdev}h\}g=\sum_LB(f,h)_L(\mlambda+\mdev)^Lg=\sum_LB(f,h)_L(\lambda_1+\dev_1)^{l_1}\cdots(\lambda_d+\dev_d)^{l_d}g$. Similarly, the skewsymmetry property is $\sum_L B(g,f)_L\mlambda^L=-\sum_L(-\mlambda-\mdev)^L B(f,g)_L$.

The set of axioms for the PVA translates into a practical formula that gives the bracket between two elements of $\cA$ in terms of the bracket between the generators $u^i$, and that it is referred to as the \emph{master formula} \cite{BdSK09}.
\begin{equation}\label{eq:master}
  \{f_{\mlambda}g\}=\sum_{\substack{i,j=1\ldots,N\\L,M\in\Z^D_{\geq 0}}}\frac{\dev g}{\dev u^j_M}(\mlambda+\mdev)^M\{u^i_{\mlambda+\mdev}u^j\}(-\mlambda-\mdev)^L\frac{\dev f}{\dev u^i_L}.
\end{equation} 

The relation between the notion of Poisson Vertex Algebra and the formal variational calculus is given by an isomorphism between the Poisson Vertex Algebras and the Poisson bracket on the space of local functionals \cite[Theorem 2 and 3]{MC15-1}. This means that working in the category of PVAs is equivalent to working with Poisson brackets in the language of the formal calculus of variation; at the same time, it is worthy noticing that in the axioms and in the formulas of Poisson Vertex Algebras there are not integrations to be performed, while all the computations that will be required in dealing with our problem can be carried out by the direct application of Formula \eqref{eq:master}, which is easily coded in a Computer Algebra System. In particular, given a Poisson bracket in the space of the local densities $\cA$ 
$$\{u^i(x),u^j(y)\}=\sum_L B^{ij}_L(u(x);u_M)\mdev^L\delta(x-y),$$ we can define the $\lambda$ bracket between the generators of $\cA$, corresponding to the coordinates on the manifold $M$, by
\begin{equation}\label{eq:defLambdaB}
 \{u^i_{\mlambda}u^j\}=\sum_S B^{ji}_S(u;u_M)\mlambda^S
\end{equation}
and extend it to the full space by means of the master formula.

Hence, the multidimensional Poisson Vertex Algebra corresponding to a Poisson bracket of hydrodynamic type as \eqref{eq:mHYPB-def} is defined by the $\lambda$ bracket
\begin{equation}\label{eq:mHYPB-defPVA}
\{u^i{}_{\mlambda}u^j\}=\sum_{\alpha=1}^d \left(g^{\alpha j i}\lambda_\alpha+b^{\alpha j i}_k u^k_\alpha\right).
\end{equation}

\subsection{Two-dimensional Poisson brackets of hydrodynamic type}
In this paragraph we specialise the previous definitions to the case of two-dimensional, two-components brackets of hydrodynamic type. Moreover, we recall the classification by Ferapontov and collaborators \cite{FOS11}, which we consider as the starting point of the following discussion.

For simplicity, we take as generators of $\cA$ the two coordinates $(p^1,p^2)=(p,q)$ and denote the two independent variables $(x^1,x^2)$ as $(x,y)$. The derivations with respect to $(x,y)$ and the corresponding jet variables will be denoted by $\dev_x^m\dev_y^n p=p_{mxny}$.

Explicitly, the conditions that $(g^x,g^y,b^x,b^y)$ must satisfy are
\begin{subequations}\label{eq:Mokhov}
\begin{gather}\label{eq:M1}
g^{\alpha ij}=g^{\alpha ji},\\\label{eq:M2}
\frac{\dev g^{\alpha ij}}{\dev p^k}=b^{\alpha ij}_k+b^{\alpha ji}_k,\\\label{eq:M3}
\sum_{(\alpha,\beta)}\left(g^{\alpha ai}b^{\beta jk}_a-g^{\beta aj}b^{\alpha ik}_a\right)=0,\\ \label{eq:M4}
\sum_{(i,j,k)}\left(g^{\alpha ai}b^{\beta jk}_a-g^{\beta aj}b^{\alpha ik}_a\right)=0,\\ \label{eq:M5}
\sum_{(\alpha,\beta)}\left[g^{\alpha ai}\left(\frac{\dev b^{\beta jk}_a}{\dev p^r}-\frac{\dev b^{\beta jk}_r}{\dev p^a}\right)+b^{\alpha ij}_ab^{\beta ak}_r-b^{\alpha ik}_ab^{\beta aj}_r\right]=0,\\ \label{eq:M6}
g^{\beta ai}\frac{\dev b^{\alpha jk}_r}{\dev p^a}-b^{\beta ij}_ab^{\alpha ak}_r-b^{\beta ik}_ab^{\alpha ja}_r=g^{\alpha aj}\frac{\dev b^{\beta ik}_r}{\dev p^a}-b^{\alpha ja}_ab^{\beta ak}_r-b^{\alpha jk}_ab^{\beta ia}_r,\\ \notag
\frac{\dev}{\dev p^s}\left[g^{\alpha ai}\left(\frac{\dev b^{\beta jk}_a}{\dev p^r}-\frac{\dev b^{\beta jk}_r}{\dev p^a}\right)+b^{\alpha ij}_ab^{\beta ak}_r-b^{\alpha ik}_ab^{\beta aj}_r\right]\\ \label{eq:M7}
+\frac{\dev}{\dev p^r}\left[g^{\beta ai}\left(\frac{\dev b^{\alpha jk}_a}{\dev p^s}-\frac{\dev b^{\alpha jk}_s}{\dev p^a}\right)+b^{\beta ij}_ab^{\alpha ak}_s-b^{\beta ik}_ab^{\alpha aj}_s\right]\\\notag
+\sum_{(i,j,k)}\left[b^{\beta ai}_r\left(\frac{\dev b^{\alpha jk}_s}{\dev p^a}-\frac{\dev b^{\alpha jk}_a}{\dev p^s}\right)\right]+\sum_{(i,j,k)}\left[b^{\alpha ai}_s\left(\frac{\dev b^{\beta jk}_r}{\dev p^a}-\frac{\dev b^{\beta jk}_a}{\dev p^r}\right)\right]=0.
\end{gather}\end{subequations}
The greek indices $(\alpha,\beta)$ can take the values $(x,y)$ and the symbol $\sum_{(a,b,c)}$ means the sum over the cyclic permutation of the indices.

The two-dimensional, two-components brackets of hydrodynamic type that fulfil the conditions \eqref{eq:Mokhov} have been classified by Ferapontov, Odesskii and Stoilov up to Miura transformations (namely, change of coordinates in the space $\cA$) and linear change of the independent variables $(x,y)$ \cite{FOS11}. They are of three types, for which we write here the $\lambda$ brackets of the corresponding PVAs.
\begin{align}\label{eq:P1}
\{p^i{}_{\mlambda}p^j\}_1&=\begin{pmatrix} \lambda_x & 0 \\ 0 &\lambda_y\end{pmatrix}\\\label{eq:P2}
\{p^i{}_{\mlambda}p^j\}_2&=\begin{pmatrix} 0 & \lambda_x \\ \lambda_x &\lambda_y\end{pmatrix}\\ \label{eq:PLP}
\{p^i{}_{\mlambda}p^j\}_{LP}&=\begin{pmatrix} 2p & q \\ q &0\end{pmatrix}\lambda_x+\begin{pmatrix} 0 & p \\ p &2q\end{pmatrix}\lambda_y+\begin{pmatrix} p_x & q_x \\ p_y &q_y\end{pmatrix}
\end{align}
In particular, it is worthy noticing that the first two structure are constant ones; this means that all the nonconstant (or equivalently, that cannot be brought to constant form) two-dimensional two-components Poisson brackets of hydrodynamic type are equivalent to \eqref{eq:PLP} -- a structure introduced by Novikov \cite{N82} and called \emph{Lie--Poisson bracket of hydrodynamic type}.
\section{Dispersive deformations in the language of Poisson Vertex Algebras}
Let us introduce on $\R[\mlambda]\otimes\cA$ a grading corresponding to the order of the differential operator acting on the Dirac's delta corresponding to the $\lambda$ bracket. This means that $\deg\dev_\alpha=1$, $\deg\lambda_\alpha=1$, and $\deg p^i_{mxny}=m+n$. The latter identity provides the standard grading on $\cA$, which is inherited by $\cF$ since the total derivatives are homogeneous operators of degree 1.
\begin{deff}A A \emph{$n$-th order dispersive deformation} of a PVA $(\cA,\{\cdot_{\mlambda}\cdot\})$ is a PVA defined by a deformed $\lambda$-bracket
\begin{equation}\label{eq:defoLambda_def}
 \{\cdot_{\mlambda}\cdot\}^\sim=\{\cdot_{\mlambda}\cdot\}+\sum_{k=1}^n\epsilon^k\{\cdot_{\mlambda}\cdot\}_{[k]}
\end{equation}
such that $\{\cdot_{\mlambda}\cdot\}^\sim$ is PVA-skewsymmetric and the PVA-Jacobi identity holds up to order $n$, namely
\begin{equation*}
 \{f_{\mlambda}\{g_{\mmu}h\}\}-\{g_{\mmu}\{f_{\mlambda}h\}\}-\{\{f_{\mlambda}g\}_{\mlambda+\mmu}h\}=O(\epsilon^{n+1}).
\end{equation*}
We require that the deformed bracket is homogeneous when we assign to the deformation parameter $\epsilon$ the degree $-1$. Moreover, we call such a deformation a \emph{finite deformation} if the PVA-Jacobi identity is satisfied to the order $2n$, namely if the deformed bracket is the $\lambda$ bracket of a PVA.
\end{deff}

A general Miura type transformation is a change of coordinates on $\cA$, compatible with the derivations. We can define the \emph{Miura group} as the group of transformations of form
\begin{equation}\label{eq:Miura}
 \begin{split}
  &p^i\mapsto \tilde{p}^i=\sum_{k=0}^\infty\epsilon^kF^i_{[k]}(p;p_{mxny})\quad m+n\leq k\\
&F^i_{[k]}\in\cA ,\quad\deg F^i_{[k]}=k\\
&\det\left(\frac{\dev F^i_{[0]}(p)}{\dev p^j}\right)\neq0.
 \end{split}
\end{equation}
In particular, we are interested in the so-called \emph{Miura transformations of the second kind}, that are the ones where $F^i_{[0]}=p^i$.

A deformation is said to be \emph{trivial} if there exists an element $\phi_\epsilon$ of the group \eqref{eq:Miura} which pulls back $\{\cdot_{\mlambda}\cdot\}^\sim$ to $\{\cdot_{\mlambda}\cdot\}$,
\begin{equation*}
 \{\phi_\epsilon(a)_{\mlambda}\phi_\epsilon(b)\}=\phi_\epsilon\left(\{a_{\mlambda}b\}\right),\qquad\forall a,b\in\cA.
\end{equation*}
We notice that such an element of the Miura group, if it exists, must be of the second kind. Equivalently, a deformation of order $n$ is trivial if there exists an evolutionary vector field $X$, namely a derivation of $\cA$ commuting with the total derivatives, such that $\{u^i{}_{\mlambda}u^j\}^\sim-\{u^i{}_{\mlambda}u^j\}=\{X(u^i){}_{\mlambda}u^j\}+\{u^i{}_{\mlambda}X(u^j)\}-X(\{u^i{}_{\mlambda}u^j\})+O(\epsilon^{n+1})$.

\subsection{Cohomology of a PVA and the theory of deformations}
As it happens in the standard theory of deformations, and in particular in finite dimensional Poisson geometry, we can extract a lot of informations by the cohomology associated to the undeformed bracket. The general theory of the cohomology of a PVA is presented in great detail in \cite{dSK13-2}, albeit only for the case of one-dimensional PVAs. Since the difference in the definitions in only technical, we will sketch here only the main ideas, which are shared with the classical theory of Poisson-Lichnerowicz cohomology \cite{L77}.

The cochain complex whose cohomology we are interested in is called the \emph{variational complex}; it can be proved that it is equivalent to the complex of local multivector fields \cite{DZ}. Denoting by $\Omega^k$ its levels, it can be proved that $\Omega^0\cong\cF$ and $\Omega^1\cong\Der^\dev(\cA)$, namely the the space of evolutionary vector fields. Moreover and crucially, the space of skewsymmetric $\lambda$ brackets is isomorphic to the space of local 2-vectors, and the property of satisfying the Jacobi identity of a PVA [Definition 1.6] is equivalent to the Schouten relation $[P,P]=0$ for local bivectors. Moreover, denoted by $P$ the element of the variational complex corresponding to the $\lambda$ bracket of a PVA, its adjoint action on the complex is a differential -- that we denote by $\ud_P$ -- namely it squares to 0. This allows us to define the \emph{cohomology of the PVA} $(\cA,\{\dev_\alpha\},\{\cdot{}_{\mlambda}\cdot\}_P=:P)$ as
\begin{equation}
H^k(P)=\frac{\ker \ud_P:\Omega^k\to\Omega^{k+1}}{\im \ud_P:\Omega^{k-1}\to\Omega^k}.
\end{equation}
Each of the cohomology groups $H^k$ can be filtered according to the previously defined differential grading. We will denote by $H^k_{[n]}$ the component of degree $n$ of the $k$-th cohomology group.

Let us consider the action of $\ud_P$ on evolutionary vector fields (namely, 1-cochains) and skewsymmetric $\lambda$ brackets (namely, 2-cochains). Let $X$ be an evolutionary vector field. It has the form
\begin{equation}\label{eq:evVF}
X=\sum_{i=1}^n\sum_L\left(\mdev^L X^i(u,u_M)\right)\frac{\dev}{\dev u^i_L}
\end{equation}
where $X^i\in\cA$ and $X(u^i)=X^i$. We have $\deg X=\deg X^i$. Then, we denote a skewsymmetric $\lambda$ bracket by $\{\cdot{}_{\mlambda}\cdot\}_K=:K$. We have
\begin{align}\label{eq:adPX}
\ud_P(X)(u^i,u^j;\mlambda)&=\{X^i{}_{\mlambda}u^j\}_P+\{u^i{}_{\mlambda}X^j\}_P-X\left(\{u^i{}_{\mlambda}u^j\}_P\right),\\\notag
\ud_P(K)(u^i,u^j,u^k;\mlambda,\mmu)&=\left\{u^i{}_{\mlambda}\{u^j{}_{\mmu}u^k\}_K\right\}_P-\left\{u^j{}_{\mmu}\{u^i{}_{\mlambda}u^k\}_K\right\}_P\\\label{eq:adPK}
&\quad -\left\{\{u^i{}_{\mlambda}u^j\}_K{}_{\mlambda+\mmu}u^k\right\}_P+\left\{u^i{}_{\mlambda}\{u^j{}_{\mmu}u^k\}_P\right\}_K\\\notag
&\quad-\left\{u^j{}_{\mmu}\{u^i{}_{\mlambda}u^k\}_P\right\}_K-\left\{\{u^i{}_{\mlambda}u^j\}_P{}_{\mlambda+\mmu}u^k\right\}_K.
\end{align}
From \eqref{eq:adPX} it follows that a trivial deformation of a PVA is in the image of $\ud_P$. Moreover, let us consider an homogeneous dispersive deformation of order $n$, namely a deformation of the form $\{\cdot{}_{\mlambda}\cdot\}^\sim=\{\cdot{}_{\mlambda}\cdot\}_P+\epsilon^n\{\cdot{}_{\mlambda}\cdot\}_{[n]}$. The PVA-Jacobi identity at the order $\epsilon^n$ is exactly $\ud_P([n])=0$. This means that the nontrivial deformations of this form must be a cocycle in $H^2_{n+\deg P}(P)$.

A similar argument applies if the deformation is of the general form \eqref{eq:defoLambda_def}. Here, let $k_1$ be the lowest degree of the deformation and $n$ the highest. The expression for the PVA-Jacobi identity can be expanded in powers of $\epsilon$ and will have an expansion $J=\sum_{k=k_1}^{2n}\epsilon^k J_k$, with
\begin{multline}
J_k=\sum_{l+m=k}\left(\left\{u^i{}_{\mlambda}\{u^j{}_{\mmu}u^k\}_{[l]}\right\}_{[m]}-\left\{u^j{}_{\mmu}\{u^i{}_{\mlambda}u^k\}_{[l]}\right\}_{[m]}
-\left\{\{u^i{}_{\mlambda}u^j\}_{[l]}{}_{\mlambda+\mmu}u^k\right\}_{[m]}\right.\\\left.+\left\{u^i{}_{\mlambda}\{u^j{}_{\mmu}u^k\}_{[m]}\right\}_{[l]}-\left\{u^j{}_{\mmu}\{u^i{}_{\mlambda}u^k\}_{[m]}\right\}_{[l]}
-\left\{\{u^i{}_{\mlambda}u^j\}_{[m]}{}_{\mlambda+\mmu}u^k\right\}_{[l]}\right),
\end{multline}
with the identification $\{\cdot{}_{\mlambda}\cdot\}_{[0]}=\{\cdot{}_{\mlambda}\cdot\}_{P}$. Since for a dispersion of order $n$ we require $J_k=0$ for $k_1\leq k\leq n$, the first condition to be imposed is $\ud_P([k_1])=0$. This means that the first nontrivial term of the deformation must be a cocycle in $H^2_{k_1+\deg P}(P)$. If the term of degree $k_1+\deg P$ is trivial, this means that there exists a Miura transformation that maps the deformed bracket into a bracket where that term vanishes -- and the cocycle condition must hold for the new lowest degree term of the deformation.

According to the already known results, a major difference exists between the one-dimensional and the multi-dimensional Poisson brackets of hydrodynamic type with respect to the theory of the deformations. A theorem by Getzler \cite{G01}, holding true for the full cohomology, as well as the particular results obtained independently by Degiovanni, Magri and Sciacca \cite{dGMS05} and by Dubrovin and Zhang \cite{DZ}, grants that the second cohomology group is trivial for all the one-dimensional brackets of hydrodynamic type. This means that all the dispersive deformations of such brackets can be brought back to the undeformed form by a Miura transformation of the second kind. Moreover, the vanishing of the third cohomology group guarantees that all the dispersive deformations can be extended to a finite deformation. While the results already obtained by the author for two-component two-dimensional brackets of hydrodynamic type \cite{MC15-1} did not allow to conclude anything about the triviality of the dispersive deformations of the bracket, they already showed that the first cohomology group is not vanishing in the second degree, as opposite as Getzler's result. Moreover, a careful study of one-component, two-dimensional brackets \cite{CCS15} tells that the cohomology groups ($H^{\geq1}$) are infinite dimensional, in particular revealing the normal form of the finite Poisson brackets of arbitrary differential order \cite{CCS17}.

The computations discussed in the next Section show that $H^2$ is not trivial -- in particular, for all the three two-dimensional two-components brackets of hydrodynamic type the first non-empty component is $H^2_3$, characterising the nontrivial deformations of second order.
\section{Cohomology groups for second order deformations}\label{sec:coho}
Since the differential order of the Poisson brackets of hydrodynamic type is 1, second order deformations are cocycles of degree $1+2=3$. In this section we compute the dimension of $H^2_3(P_1)$, $H^2_3(P_2)$, and $H^2_3(P_{LP})$, where the three $\lambda$ brackets are respectively the ones given in \eqref{eq:P1}, \eqref{eq:P2}, and \eqref{eq:PLP}. Moreover, we find an explicit representative for the elements in the cohomology classes.
\subsection{The general procedure}
We roughly follow the same procedure described in \cite{MC15-1}, relying on an improved algorithm for our Mathematica software \cite{CV16} and validating the results with the \texttt{CDE} package of REDUCE algebra system \cite{V10, V14}. The improved algorhytm allows to impose the skewsymmetry in a faster way and to effectively deal with a lesser number of unknown functions.

We start with a third degree skewsymmetric $\lambda$ bracket, whose general form can be obtained by the following procedure. Let $M=M^{ij}_S \mlambda^S\in \cA[\lambda_1,\lambda_2]$ a $n\times n$ matrix whose entries are formal polynomials in $\lambda$'s, with differential polynomials as coefficients. Let us consider only homogeneous polynomials, for which we assign degree 1 to $\lambda$ and use the differential grading in $\cA$, and let $M^{*ij}=(-\mlambda-\mdev)^SM^{ij}_S$. Then the $\lambda$ bracket defined on the generators by $\{p^i{}_{\mlambda}p^j\}=M^{ij}-M^{*ji}$ is skewsymmetric. Notice that the definition of the bracket makes meaningful only the symmetric part of the matrices $M_S$ where $|S|=s_1+s_2+\cdots+s_d$ is odd and the skewsymmetric one when $|S|$ is even.

We then define the third degree skewsymmetric bracket as follows

\begin{equation}\label{eq:gencoc}
\left\{{p^i}_{\mlambda}{p^j}\right\}_{[2]}=M_{ij}-M^{*}_{ji}
\end{equation}
where
\begin{equation}
\begin{split}
M_{ij}&=A^{abc}_{ij}(p)\lambda_a\lambda_b\lambda_c+B^{ab,cl}_{ij}(p)\dev_cp^l\lambda_a\lambda_b+C^{a,bl,cm}_{ij}(p)\dev_bp^l\dev_cp^m\lambda_a\\
&\quad+D^{a,bcl}_{ij}(p)\dev_{bc}p^l\lambda^a+\frac12E^{abcl}_{ij}(p)\dev_{abc}p^l+\frac12F^{abl,cm}_{ij}(p)\dev_{ab}p^l\dev_cp^m\\
&\quad+\frac12G^{al,bm,cn}_{ij}(p)\dev_ap^l\dev_bp^m\dev_cp^n 
\end{split}
\end{equation}
where summation over all the repeated indices for $i=1\ldots2$ is assumed and $\dev_a=\dev_{x^a}$.
All the coefficients enjoy some symmetry properties given by the definition. In particular, $A$'s are totally symmetric in the exchange of the indices $(a,b,c)$, $B$'s are symmetric in the exchange of $(a,b)$, $C$'s for the simultaneous exchange of $(b,l\leftrightarrow c,m)$, $D$'s are symmetric in $(b,c)$, $E$'s in $(a,b,c)$, $F$'s in $(a,b)$ and, finally, $G$'s are completely symmetric for the simultaneous exchange of $(a,l\leftrightarrow b,m \leftrightarrow c,n)$. To deal only with the independent parameters, we define $A$, $C$, and $D$ symmetric in $(i,j)$ and $B$, $E$, $F$, and $G$ skewsymmetric.

The number of free coefficients is 172. We proceed to derive the so called \emph{cocycle condition}, namely the set of algebraic and differential equations that the 172 functions must satisfy in order to fulfill, respectively, $\ud_{P_1}([2])=0$, $\ud_{P_2}([2])=0$, and $\ud_{P_{LP}}([2])=0$. The preliminary output of the software consists of systems of, respectively, 1700, 1879, and 3273 algebraic-differential equations.  coboundary conditions. These three systems are linear in the unknown coefficients, hence they can be reduced, either by directly computing their Janet basis \cite{PR05} or first solving the purely algebraic equations and then computing the Janet basis of the linear system of the remaining unknowns. The solutions of these system are the elements of the kernel of the Poisson differentials, namely compatible deformations of order two. 

We now consider the form of the trivial deformations of the same order. The Miura transformation producing them has the general form
\begin{equation}
\begin{split}
 p^i\mapsto P^i&=p^i+\epsilon^2\left(f^{abl,i}\dev_{ab}p^l +g^{albm,i}\dev_ap^l\dev_bp^m\right)\\
 &=p^i+\epsilon^2 X^i.
 \end{split}
\end{equation}
and the corresponding deformations are
\begin{equation}\label{eq:trivdefo}
\{p^i{}_{\mlambda}p^j\}_{\bullet,[2]}=\{X^i{}_{\mlambda}p^j\}_\bullet+\{p^i{}_{\mlambda}X^j\}_\bullet-X\left(\{p^i{}_{\mlambda}p^j\}_\bullet\right)
\end{equation} 
with $\bullet$, respectively, equal to $P_1$, $P_2$, and $P_{LP}$. We expand \eqref{eq:trivdefo} identifying the coefficients of a trivial deformation. These expressions can be read as a system of inhomogeneous differential equations for the parameters of the Miura transformation, in terms of the coefficients of the deformation. We call the compatibility conditions of such system of linear PDEs the \emph{coboundary condition}. They represent the conditions that the coefficients of a deformation must satisfy to be obtained by a Miura transformation, namely for the deformation to be in the image of $\ud_P$.

Hence, the cohomology classes in $H^2_3$ are identified by the solutions of the cocycle condition that fail to be solutions of the coboundary one, while it is obvious that the solutions of coboundary conditions are also solutions of the cocycle one.
\subsection{$H^2_3(P_1)$}\label{sec:H23P1}
The Janet basis of the cocycle condition for a bracket compatible with $P_1$ consists of 193 equations, while the Janet basis of the coboundary condition is made of 189 equations.

Comparing the two sets of equations, we get the coboundary conditions not satisfied by a generic cocycle, in the first column of the following table. The parameters involved in these four equations, however, must satisfy the equations in the second and third column.
\begin{align*}
 \text{\textbf{Coboundary}} &&\text{\textbf{Cocycle}}&&\\
 D^{2,221}_{11}&=0&\frac{\dev D^{2,221}_{11}}{\dev p}&=0&\frac{\dev D^{2,221}_{11}}{\dev q}&=0\\
 D^{1,112}_{22}&=0&\frac{\dev D^{1,112}_{22}}{\dev p}&=0&\frac{\dev D^{1,112}_{22}}{\dev q}&=0\\
 A^{222}_{11}&=0&\frac{\dev A^{222}_{11}}{\dev p}&=-\frac{1}{3}D^{2,221}_{11}&\frac{\dev A^{222}_{11}}{\dev q}&=0\\
 A^{111}_{22}&=0&\frac{\dev A^{111}_{22}}{\dev p}&=0&\frac{\dev A^{111}_{22}}{\dev q}&=-\frac{1}{3}D^{1,112}_{22}.
\end{align*}
 We identify a representative for the cohomology class $H^2_3(P_1)$ as a solution of the cocycle conditions -- the eight we have written and, of course, all the remaining ones -- that is not a solution of the coboundary ones, such as the following:
\begin{align}
A^{111}_{22}&=c_2 q+ \frac{c_4}{2}&D^{1,112}_{22}&=-3c_2\\\notag
A^{222}_{11}&=c_1 p + \frac{c_3}{2}&D^{2,221}_{11}&=-3c_1.
 \end{align}
The remaining coefficients in the definition of $P_{1,[2]}$ can be set to 0 as this satisfies all the remaining 185 equations shared by the cocycle and coboundary conditions.
\begin{thm}
The cohomology group $H^2_3(P_1)$ has dimension 4. A representative of an element of the cohomology class is

\begin{subequations}
\begin{align}\label{eq:cocH23P1}
\{p{}_{\mlambda}p\}_{P_1,[2]}&=\left(2c_1p+c_3\right)\lambda_2^3+3c_1p_y\lambda_2^2-3c_1p_{yy}\lambda_2-2c_1p_{yyy}\\
\{p{}_{\mlambda}q\}_{P_1,[2]}&=0\\
\{q{}_{\mlambda}p\}_{P_1,[2]}&=0\\
\{q{}_{\mlambda}q\}_{P_1,[2]}&=\left(2c_2q+c_4\right)\lambda_1^3+3c_2q_x\lambda_1^2-3c_2q_{xx}\lambda_1-2c_2q_{xxx}.
\end{align}
\end{subequations}
\end{thm}

\subsection{$H^2_3(P_2)$}
The Janet basis of the cocycle condition for a bracket compatible with $P_2$ consists of 212 equations, while the Janet basis of the coboundary condition is made of 207 equations.

As in the previous case, most of the coboundary conditions are satisfied by all the solutions of the cocycle ones. However, there are 22 linear combinations of the parameters that vanish for the \textbf{Coboundary condition} but do not for generic cocycles. They are not independent and we can select 5 of them, of which the other ones are linear combinations. We denote them $(b_0,b_1,b_2,b_3,b_4)$ and they are as follows:
\begin{align*}
b_0&:=A^{122}_{11}-\frac{2}{3}A^{222}_{12}=0\\
b_1&:=\frac{13}{18}D^{1,221}_{11} -\frac29 D^{1,222}_{12} - \frac29 D^{2,112}_{11}-\frac79D^{2,121}_{11}+\frac49D^{2,122}_{12}-\frac13D^{2,221}_{12}\\
&\quad-\frac43E^{122,2}_{12}-\frac43E^{2,221}_{12}=0\\
 b_2&:=D^{1,221}_{11} + 2D^{2,121}_{11}  - 2 D^{2,221}_{12}=0\\
b_3&:=B^{22,11}_{12} +\frac12 D^{1,112}_{11} +  D^{1,121}_{11} - 2 D^{1,122}_{12} + \frac12 D^{1,222}_{22} - 
 \frac12 D^{2,111}_{11} -  D^{2,112}_{12}+\\
&\quad +  D^{2,122}_{22} -\frac12 D^{2,221}_{22} - 3 E^{1221}_{12}=0\\
b_4&:=A^{222}_{11}=0.
\end{align*}
On the other hand, their corresponding \textbf{Cocycle condition} is
\begin{align*}
\frac{\dev b_0}{\dev p}&=\frac{1}{3}\left(b_1-2 b_2\right)&\frac{\dev b_1}{\dev q}&=0\\
\frac{\dev b_1}{\dev p}&=0&\frac{\dev b_1}{\dev q}&=0\\
\frac{\dev b_2}{\dev p}&=0&\frac{\dev b_2}{\dev q}&=0\\
\frac{\dev b_3}{\dev p}&=0&\frac{\dev b_3}{\dev q}&=0\\
\frac{\dev b_4}{\dev p}&=0&\frac{\dev b_4}{\dev q}&=0.
\end{align*}
\begin{thm}\label{thmH2P2}
The cohomology group $H^2_3(P_2)$ is a vector space of dimension $5$. A representative of an element of the cohomology class is the $\lambda$ bracket defined as follows:
\begin{subequations}
\begin{align}\label{eq:cocH23P2}\notag
\{p{}_{\mlambda}p\}_{P_2,[2]}&=c_1\left(2\lambda_1\lambda_2^2p+2\lambda_1\lambda_2p_y+\lambda_2^2p_x+3\lambda_1p_{yy}+2\lambda_2p_{xy}+2p_{xyy}\right)+\\
&\quad+2c_4\lambda_2^3,\\\notag
\{p{}_{\mlambda}q\}_{P_2,[2]}&=c_1\left(2\lambda_2^2q_x+2\lambda_1\lambda_2q_y+\lambda_1q_{yy}+\lambda_2q_{yy}+2\lambda_2q_{xy}+2q_{xyy}\right)\\
&\quad+c_2\left(2\lambda_2^3p+4\lambda_2^2p_y\right) -3c_0\lambda_2^3,\\\notag
\{q{}_{\mlambda}p\}_{P_2,[2]}&=-c_1\left(2\lambda_2^2q_x+2\lambda_1\lambda_2q_y+\lambda_1q_{yy}-\lambda_2p_{yy}+2\lambda_2q_{xy}+2q_{xyy}\right)\\
&\quad+c_2\left(2\lambda_2^3p+2\lambda_2^2p_y-2\lambda_2p_{yy}-2p_{yyy}\right) -3c_0\lambda_2^3,\\\notag
\{q{}_{\mlambda}q\}_{P_2,[2]}&=-c_1\left(2\lambda_2q_{yy}-q_{yyy}\right)-2c_2\left(2\lambda_2q_{yy}+q_{yyy}\right)\\
&\quad+\frac{2c_3}{3}\left(2\lambda_2^3p+3\lambda_2^2p_y-3\lambda_2p_{yy}-2p_{yyy}\right).
\end{align}
\end{subequations}
\end{thm}
\begin{proof}
From the Cocylce conditions previously listed, we easily conclude that $H^2_3(P_2)\cong \R^5$. Indeed, a solution of the cocycle conditions that is not a solution of coboundary conditions depends on five constants, since we have
\begin{subequations}\label{eq:H23P2}
\begin{align}
b_0&=\frac{1}{3}\left(c_1-2c_2\right)p + c_0\\
b_1&=c_1&b_2&=c_2\\
b_3&=c_3&b_4&=c_4
 \end{align}
 \end{subequations}
with $(c_1,\ldots,c_5)$ constants. In addition to \eqref{eq:H23P2}, the coefficients of a generic cocycle must satisfy all the remaining conditions. Their form suggests that we can look for a representative in which the coefficients $A$ are at most linear in $p$ and all the other ones are constant. The cocycle conditions turn into a set of algebraic equations, for which we easily found one solution as follows
\begin{align*}
A^{122}_{11}&=\frac{1}{3}c_1p&A^{222}_{11}&=c_4\\
A^{222}_{12}&=c_2p-\frac32c_0&A^{222}_{22}&=\frac23c_3p\\
B^{12,22}_{12}&=\frac{c_1}{2}&B^{22,12}_{12}&=\frac{c_1}{2}\\
B^{22,21}_{12}&=\frac{c_2}{2}&&\\
D^{1,221}_{11}&=c_1&D^{2,221}_{12}&=\frac12c_1-2c_2\\
D^{2,221}_{22}&=-2c_3&D^{2,222}_{22}&=-c_1-2c_2\\
E^{122,2}_{12}&=\frac{c_1}{6}&E^{222,1}_{12}&=\frac12(c_2-c_1),
\end{align*}
and the remaining coefficients are set to vanish. This gives the explicit form for a generic compatible bracket.
\end{proof}
\subsection{$H^2_3(P_{LP})$}\label{sec:H23PLP}
The set of equations for the cocycle condition consists here of 3723 among algebraic equations and linear PDEs. The computers available to us cannot process such a huge number of equations to provide a Janet basis, so we first solve algebraically for 128 of the 172 coefficients of the deformed bracket, getting a system of 1706 linear PDEs for 34 variables, namely $A^{abc}_{ij}$, $B^{ab,cl}_{12}$, and $D^{2,111}_{22}$, $D^{1,222}_{22}$, $D^{1,111}_{22}$, $D^{2,222}_{12}$, $D^{2,111}_{12}$, $D^{1,221}_{12}$, $D^{1,111}_{12}$, $D^{2,222}_{11}$, $D^{2,111}_{11}$, and $D^{1,222}_{11}$. The Janet basis of this reduced system is constituted by 6 equations, and together with all the algebraic relations is the \textbf{Cocycle condition}.

We then proceed to compute the compatibility conditions of the nonhomogeneous system of PDEs expressing the aforementioned 34 coefficients in terms of the parameters of a generic Miura transformation. Its Janet basis consists of 4 equations that are provided in the Appendix (Equations \eqref{eq:JComp1}--\eqref{eq:JComp4}). Out of the 6 Cocycle conditions, two are equivalent to the Coboundary condition \eqref{eq:JComp1} and \eqref{eq:JComp2}. By taking successive derivatives of the remaining two coboundary conditions, we observe the following
 \begin{align*}
  \text{\textbf{Coboundary}}&&\text{\textbf{Cocycle}}&\\
  \text{LHS.\eqref{eq:JComp3}}&=0&\text{LHS.\eqref{eq:JComp3}}&=2c_1 p^2q^3+30c_2 p^3q^2\\
 \text{LHS.\eqref{eq:JComp4}}&=0&\text{LHS.\eqref{eq:JComp4}}&=-2c_1 p^2q^3-198c_2 p^3q^2.
 \end{align*}
 \begin{thm}
 The cohomology group $H^2_3(P_{LP})$ has dimension 2. A generic representative in $H^2_3(P_{LP})$  has the form $\{p^i{}_{\mlambda}p^j\}_{P_{LP},[2]}=c_1\{p^i{}_{\mlambda}p^j\}^{(1)}+c_2\{p^i{}_{\mlambda}p^j\}^{(2)}$. We write here the leading order terms of the elements in the basis of the cohomology class and leave the full expression to the appendix \ref{app:H23PLP}
\begin{subequations}\label{eq:cocH23_1}
\begin{align}
\{p{}_{\mlambda}p\}^{(1)}&= 2\lambda_1^3  +O(\lambda_1,\lambda_2)\\
\{p{}_{\mlambda}q\}^{(1)}&=2\lambda_1^2\lambda_2+ O(\lambda_1,\lambda_2)\\
\{q{}_{\mlambda}p\}^{(1)}&=2\lambda_1^2\lambda_2  O(\lambda_1,\lambda_2)\\
\{q{}_{\mlambda}q\}^{(1)}&=2\lambda_1\lambda_2^2+ O(\lambda_1,\lambda_2),
\end{align}
\end{subequations}
and
\begin{subequations}\label{eq:cocH23_2}
\begin{align}
\{p{}_{\mlambda}p\}^{(2)}&= O(\lambda_1,\lambda_2)\\
\{p{}_{\mlambda}q\}^{(2)}&=-\lambda_2^2\left(\frac{p_x}{p}-\frac{11p_y}{2q}\right) + O(\lambda_1,\lambda_2)\\
\{q{}_{\mlambda}p\}^{(2)}&=\lambda_2^2\left(\frac{p_x}{p}-\frac{11p_y}{2q}\right)+ O(\lambda_1,\lambda_2)\\
\{q{}_{\mlambda}q\}^{(2)}&=O(\lambda_1,\lambda_2).
\end{align}
\end{subequations}
 \end{thm}
 \begin{proof}
The solution of the cocycle conditions that are not solution of the coboundary ones depends on two arbitrary constants, that gives us the dimension of the cohomology group. As an illustration of the form of the cocycles, we choose two solutions for the system of 34 coefficients. The ``normalisation'' factors in front of the constants $(c_1,c_2)$ in the previous set of equations are chosen to get a simpler form for the representatives in the cohomology class. We have, respectively,
\begin{align*}
 \begin{cases}
  A^{111}_{11}&=1\\A^{112}_{12}&=\frac13\\A^{122}_{22}&=\frac13\\\text{other 31 coeff.}&=0
 \end{cases}
&&\text{and}&& \begin{cases}
  B^{22,21}_{12}&=\frac{11}{2q}\\B^{22,11}_{12}&=-\frac1p\\\text{other 32 coeff.}&=0
 \end{cases}
\end{align*}
To get the full form of the two cocycles we then need to use the identities for the remaining 128 coefficients of the skewsymmetric bracket. 
\end{proof}
\subsection{Further remarks on $H^p_d(P_1)$}
It is worthy noticing an interesting pattern in the dimension of the cohomology groups $H^p_d(P_1)$, as explicitly computed in several examples in \cite{MC15-1}, the Paragraph \ref{sec:H23P1} of the current paper, and \cite{MC15-2}.

We can compare it with the dimension of the cohomology groups for the scalar two-dimensional bracket $\{u{}_{\mlambda}u\}_{(sc)}=\lambda_y$ computed in \cite{CCS15} for general values of $(p,d)$. We get
\begin{center}
\begin{tabular}{|r | *{7}{c}|}
\hline
$(p,d)$ & (1,0) & (1,1) &(1,2) & (2,1) & (2,2) & (2,3) & (3,2)\\
\hline
$\dim H^p_d(P_{(sc)})$ & 1 & 1 & 0 & 1 & 0 & 2 & 0\\
\hline
$\dim H^p_d(P_1)$ & 2 & 2 & 0 & 2 & 0 & 4 & 0\\
\hline
\end{tabular}
\end{center}
from which we immediately observe that the dimension of the cohomology groups for $P_1$ appears, in all the cases we have computed, to be twice the dimension of the corresponding cohomology for the scalar bracket.

The reason of this behaviour is the diagonal form of the bracket $P_1$. Indeed, $P_1\cong P_{(sc)}(\cF_p)\oplus P_{(sc)}(\cF_q)$, where by $\cF_p$ and $\cF_q$ we denote, respectively, the space of local functionals whose densities depend only on $p(x,y)$ and $q(x,y)$. This means that the theorem proved in \cite{CCS15} for the dimension of the cohomology groups of a scalar multidimensional bracket, specialised to the case where $d=2$, can be used to determine the dimension of the full (infinite, as expected) cohomology of $P_1$.

The dimension can be conveniently written using a generating function \cite{CCS15}
\begin{equation}\label{eq:dimcoh}
\sum_{d\geq0}\dim H^p_d(P_1)x^d=2\left(h^p(x)+h^{p+1}(x)\right)
\end{equation}
where
\begin{equation}
h^p(x) := x^{\frac{p}2 (p-1)} \prod_{i=2}^p (1-x^i)^{-1}  +x \delta_{p,0}.
\end{equation}
We provide a few explicit values for the dimension of $H^p_d(\cF)$ obtained by \eqref{eq:dimcoh} in the following table.
\begin{center}
\begin{tabular}{| c | *{10}{c} |}
\hline
 $d\rightarrow$ & 1 & 2 & 3 & 4 & 5 &6&7&8&9&10\\
\hline
$p=1$ & 2& 0 & 2 & 0 & 2 &0 &2 &0 & 2 & 0\\
\hline
$p=2$ & 2& 0 & 4 & 0 &4 &2 &4&2&6&2\\
\hline
$p=3$ & 0 & 0 & 2 &0 &2 &4&2 & 4&6&6\\
\hline
\end{tabular}
\end{center}
\section{Finite deformations and obstructions}
By definition of PVA cohomology, the elements in $H^2_3(P)$ are classes, under Miura transformations, of $\lambda$-brackets of degree 3 compatible with $\{\cdot{}_{\mlambda}\cdot\}_P$ that are not themselves Miura transformed of $P$. In our previous paper \cite{MC15-1}, it was proved that $H^2_2(P)=0$ for all the three classes of two-components, two-dimensional $\lambda$ brackets. This means that, according to the definition given in Section 2, a \emph{second order deformation} of, respectively, $P_1$, $P_2$, and $P_{LP}$ is given, up to Miura transformation, by the bracket itselft plus one element of the corresponding $H^2_3(P)$.

In general, such a deformation may be finite, if the resulting bracket satisfies the PVA-Jacobi identity. A deformation which is not finite can sometimes be \emph{extended} to an infinitesimal deformation of higher order. Let $P_{0}$ denote, symbolically, the undeformed $\lambda$ bracket of a Poisson vertex algebra. For the sake of compactness, in this section we will adopt a ``Schouten bracket'' notation for the PVA-Jacobi identity, namely
\begin{multline}\label{eq:defsch}
[A,B]^{ijk}_{\mlambda\mmu}:=\left\{u^i_{\mlambda}\left\{u^j{}_{\mmu}u^k\right\}_{B}\right\}_A+\left\{u^i_{\mlambda}\left\{u^j{}_{\mmu}u^k\right\}_{A}\right\}_B-\left\{u^j_{\mmu}\left\{u^i{}_{\mlambda}u^k\right\}_{B}\right\}_A\\-\left\{u^j_{\mmu}\left\{u^i{}_{\mlambda}u^k\right\}_{A}\right\}_B
-\left\{\left\{u^i{}_{\mlambda}u^j\right\}_A{}_{\mlambda+\mmu}u^k\right\}_B-\left\{\left\{u^i{}_{\mlambda}u^j\right\}_B{}_{\mlambda+\mmu}u^k\right\}_A,
\end{multline}
in such a way that we can denote the PVA-Jacobi identity for $P_{0}$ as $[P_{0},P_{0}]=0$, where we dropped all the indices because the identity holds for all $(i,j,k)$ and choice of the formal parameters $(\mlambda,\mmu)$. This notation is not only convenient, but motivated by the identification between $\lambda$ brackets of a PVA and local Poisson bivectors, obtained by the isomorphism of both with the Lie algebra structure on the space of local differentials \cite{dSK13-2}.

Let $P_{[1]}$ be an element in the first non-vanishing $H^2(P_{0})$ cohomology group, say $P_{[1]}\in H^2_k$. It defines a ($k$-th order) infinitesimal deformation of $P_{(0)}$, since
\begin{equation*}
\begin{split}
[P_{0}+\epsilon^kP_{[1]},P_{0}+\epsilon^kP_{[1]}]&=[P_{0},P_{0}]+2\epsilon^k[P_{0},P_{[1]}]+\epsilon^{2k}[P_{[1]},P_{[1]}]\\
&=\epsilon^{2k}[P_{[1]},P_{[1]}],
\end{split}
\end{equation*}
where the first term vanishes because $P_{0}$ is a PVA bracket and the second because it coincides with $\ud_{P_{0}}P_{[1]}=0$.

If $P_{[1]}$is not the bracket of a PVA, however, such a deformation is not finite. We call such a deformation \emph{extendible} if there exists a $\lambda$ bracket $P_{[2]}$, of order $2k+\deg P_{0}$, such that the equation
\begin{equation}\label{eq:extend}
2[P_{0},P_{[2]}]+[P_{[1]},P_{[1]}]=0
\end{equation}
has solution. If such is the case, indeed, the bracket $P_{0}+\epsilon^kP_{[1]}+\epsilon^{2k}P_{[2]}$ is a deformation of order $2k$. In the case we have just illustrated, $[P_{[1]},P_{[1]}]$ is a 3-cocycle, and the Equation always admits a solution only if $H^3_{2k}=0$. A nonvanishing third cohomology group in the appropriate degree is, hence, an \emph{obstruction} to the extension. This means that in the 1-dimensional case, where Getlzer's result holds true, all the infinitesimal deformations are extendible.

In this Section, rather than attempting the computation of the third cohomology group at the appropriated degree for the three PVA structures $P_1$, $P_2$, and $P_{LP}$, we will directly look for solutions of the equation \eqref{eq:extend}, with $P_{0}$ one of the three brackets of hydrodynamic type and $P_{[1]}$ elements of the corresponding $H^2_3$ group.
\begin{rmk}
Equation \eqref{eq:extend} is not linear in $P_{[1]}$, so it does not allow to separately consider the elements of a basis in $H^2_3(P_{0})$, whose dimension let us denote $n$. Chosen a basis $\{P_{[1]}^{(a)}\}_{a=1}^n$ for the cohomology group, we can write $P_{[1]}=\sum_{a=1}^n c_a P_{[1]}^{(a)}$ and Equation \eqref{eq:extend} as
\begin{equation}\label{eq:extendsplit}
2[P_{0},P_{[2]}]+\sum_{a,b=1}^nc_ac_b[P_{[1]}^{(a)},P_{[1]}^{(b)}]=:C(P_{0},P_{[1]},P_{[2]})=0.
\end{equation}
It would be of little interest to consider the $n(n+1)/2$ possible matching and looking for the solvability of \eqref{eq:extend} in each case. In this Section we will, then, only consider the extendability of generic elements of $H^2_3$ (namely, $c_a\neq 0\,\forall a$). Moreover, we will discuss whether the elements of the cohomology basis are \emph{bona fide} brackets of a Poisson vertex algebra, namely if they satisfy the PVA-Jacobi identity themselves.
\end{rmk}
\begin{thm}
The second order infinitesimal deformations of, respectively, $P_1$, $P_2$ and $P_{LP}$ given by generic elements of the corresponding group $H^2_3$ are \emph{not extendible}.
\end{thm}
\begin{proof}
The theorem follows from explicit computations. In particular, the ones for $H^2_3(P_{LP})$ have been performed using REDUCE Algebra system with the packages \textsl{CDE} \cite{V14} and \textsl{CDiff}\cite{V10}, that rely on the so-called $\theta$ formalism \cite{CCS15}. The resulting set of equation has then been analysed using the Maple package Janet \cite{PR05}. We used Mathematica computer algebra system with the package \textsl{MasterPVAmulti} \cite{CV16}, developed to perform computations with the $\lambda$ bracket formalism, for the other two cases.

We want to solve equation \eqref{eq:extendsplit} for $P_{0}$ respectively $P_1$, $P_2$, and $P_{LP}$ (or $P_{\bullet}$ for short) and where the second term is given by (twice) the PVA-Jacobi identity for the corresponding representative of the cohomology classes computed in Section \ref{sec:coho}. We write a candidate $P_{[2]}$, that we prefer in this context to denote by $P_{[4]}$ since it would correspond to the 4th order term of the deformation of the bracket of hydrodynamic type, being a homogeneous degree 5 skewsymmetric $\lambda$ bracket. We write the candidate $P_{[4]}$, which depends on 1774 parameters, similarly to what we did in Section 3.1, by defining it as
$$
\{p^i{}_{\mlambda}p^j\}_{[4]}=N_{ij}-N^*_{ij},
$$
with $N_{ij}$ a homogeneous differential polynomial of fifth degree in $\lambda$ and in the jet variables. The second term in Equation \eqref{eq:extend} is $[P_{[1]},P_{[1]}]$. Again, we stick to the naming convention adopted in Section \ref{sec:coho} and denote it by $P_{\bullet,[2]}$ to keep track of the corresponding hydrodynamic structure and of the deformation order. Notice that the absence of $c_a^2$ terms in $[P_{\bullet,[2]},P_{\bullet,[2]}]$ means that the corresponding bracket $P_{\bullet,[2]}^{(a)}$ is the bracket of a PVA, while the lack of $c_ac_b$ terms means that the corresponding elements of the basis are compatible. We then split the full expression collecting powers of $\lambda$'s, $\mu$'s, and of the partial derivatives up to the 6th order, and try to equate to 0 all the terms, solving for some of the many free parameters of $P_{[4]}$. 
\begin{description}
\item{Case $P_1$}
In the expansion of $[P_{1,[2]},P_{1,[2]}]$ we only have the terms $c_1^2$, $c_2^2$, $c_1c_3$, $c_2c_4$. This is enough to conclude that $P_{1,[2]}^{(3)}$ and $P_{1,[2]}^{(4)}$ are compatible brackets of a Poisson vertex algebra. Moreover, in the coefficients of $\lambda_1^5\mu_1$ of $C(P_1,P_{1,[2]},P_{[4]})^{222}$ and $\lambda_2^5\mu_2$ of $C(P_1,P_{1,[2]},P_{[4]})^{111}$ (the upper indices correspond to the component as from Equation \eqref{eq:defsch}) we observe the lack of terms coming from $[P_1,P_{[4]}]$, while the whole coefficients are not vanishing because of terms coming from, respectively, $\left[P_{1,[2]}^{(2)},P_{1,[2]}^{(2)}\right]$, $\left[P_{1,[2]}^{(2)},P_{1,[2]}^{(4)}\right]$ and $\left[P_{1,[2]}^{(1)},P_{1,[2]}^{(1)}\right]$, $\left[P_{1,[2]}^{(1)},P_{1,[2]}^{(3)}\right]$ that cannot vanish except when $c_1$ and $c_2$ are 0. 
\item{Case $P_2$}
The expansion of $[P_{2,[2]},P_{2,[2]}]$ contains terms $c_ac_b$ for $a=1,2$ and $b=0,\ldots,4$, as well as the terms $c_3c_4$ and $c_0c_3$. In the full expression $C(P_2,P_{2,[2]},P_{[4]})^{112}$ we observe that the coefficient of $\lambda_2^5\mu_2$ is $2c_1c_4$, and the one of $\lambda_2^4\mu_2^2$ is $8c_2c_4$, that cannot vanish since there are not corresponding term coming from $[P_2,P_{[4]}]$. Therefore the deformation $P_{2,[2]}$ is not extendible.
\item{Case $P_{LP}$}
The expansion of $[P_{LP,[2]},P_{LP,[2]}]$ contains all the terms $c_ic_j$ for $i,j=1,2$. Computing the full expression $C(P_{LP},P_{LP,[2]},P_{[4]})$ with the \textsl{CDE} package of Reduce computer algebra system, we obtain a huge overdetermined set of linear PDEs for the unknown parameters of $P_{[4]}$. None of this is impossible at sight as in the previous cases. However, we can select a subset of linear PDEs, corresponding to the coefficients of $\lambda_1^m\lambda_2^n\mu_1^s\mu_2^{6-m-n-s}$ (in fact, coefficients of $\theta_i\theta_j^m\theta_k^{6-m}$, since \textsl{CDE} uses the $\theta$ formalism). This subset of equation is significantly simpler, being constituted by 248 equations, a lot of which purely algebraic, for 206 equations; such a system, as its parent one, is nonhomogenous because of the presence of the terms coming from $[P_{LP,[2]},P_{LP,[2]}]$. The package Janet provides tools for computing the compatibility condition of systems of linear PDEs -- the only possible solution of which being $c_1=c_2=0$.
\end{description}
\end{proof}
\begin{cor}\label{cor5}
The elements of the cohomology group $H^2_3(P_1)$ in the class identified by $(c_1,c_2,c_3,c_4)=(0,0,c_3,c_4)$ are $\lambda$ brackets of Poisson vertex algebras. In particular, this means that the third and fourth elements of the basis of $H^2_3(P_1)$ form a biHamiltonian pair.

The elements of the cohomology group $H^2_3(P_2)$ given by $(c_0,c_1,c_2,c_3,c_4)$ equal, respectively, to $(0,0,0,c_3,0)$ and $(c_0,0,0,0,c_4)$ are $\lambda$ brackets of Poisson vertex algebras. In particular, $P_{2,[2]}^{(0)}$ and $P_{2,[2]}^{(4)}$ form a biHamiltonian pair.
\end{cor}
\begin{cor}
The degree 6 components of the third cohomology groups $H^3_6(P_1)$, $H^3_6(P_2)$, $H^3_6(P_{LP})$ for  the two-dimensional two-components Poisson brackets of hydrodynamic type are not trivial.
\end{cor}
\begin{eg}[A bi-Hamiltonian pair for a simple PDE]
Let us consider the Poisson pencil $ P_{2,[2]}^{(3)}+\lambda P_2$. The compatibility between $P_{2,[2]}^{(3)}$ and $P_2$ is guaranteed by Theorem \ref{thmH2P2} and $P_{2,[2]}^{(3)}$ is a Poisson bracket itself by Corollary \ref{cor5}. Explicitly, it has the form
\begin{equation}\label{eq:eg-pencil}
P_{2,[2]}^{(3)}+\lambda P_2=\begin{pmatrix}0 &0 \\ 0 & 2p\dev_y^3+3p_y\dev_y^2-3p_{yy}\dev_y-2p_{yyy}\end{pmatrix}+\lambda \begin{pmatrix}0 & \dev_x\\\dev_x&\dev_y\end{pmatrix}.
\end{equation}
It follows from $\ker \dev_x=\ker\dev_y=\R$ that the Casimirs for $P_2$ are the Hamiltonian functionals $\int p$ and $\int q$. We pick $H_0=\int q$ because it is the only one which is not a Casimir of $P_{2,[2]}^{(3)}$, too. $P_{2,[2]}^{(3)}\delta H_0$ yields the Hamiltonian system
\begin{equation}
\begin{cases}\label{eq:eg-flow}
p_t=0,&\\
q_t=-2\, p_{yyy}.
\end{cases}
\end{equation}
Such a system is biHamiltonian. Indeed, it can be written as $P_2\delta H_1$ for the nonlocal Hamiltonian functional
\begin{equation}
H_1=\int p\left(\dev_x^{-1}p_{yyy}\right).
\end{equation}
However, $H_1$ is a Casimir of $P_{2,[2]}^{(3)}$, so we have a so-called short Lenard chain for the system \eqref{eq:eg-flow}.
\end{eg}
\section{Concluding remarks}
In this paper we have discussed the second order deformations of two-dimensional, two-component brackets of hydrodynamic type, with a mainly computational approach. This has been enough to conclude that, already for this case, the compatible deformation of the Poisson brackets are not Miura equivalent to the hydrodynamic type ones, as opposite as what is the case for one-dimensional, any-component case. On the contrary, there is a wide class -- arguably wider as the order of deformations increases, as it happens in the one-component case \cite{CCS15} -- of not equivalent and not trivial deformations. It is interesting, however, that none of the deformations we have found is actually extendible to higher order. The extension is either obstructed or not necessary, because it is already Poisson. It is worthy noticing, however, that we have found only one deformation ($P_{2,[2]}^{(3)}$) which is Poisson without being constant, that is the case that trivially fulfils Jacobi identity.

The technique we have used for these results can be extended to higher order deformations, where we are practically limited by the computational time and by the memory of the machine. However, to obtain a general result it seems necessary to overcome the order-by-order approach. In a paper together with G.~Carlet and S.~Shadrin we found the expression for the dimension of any cohomology group in the scalar case \cite{CCS15}. This result can be used to compute, for any order of deformation, the dimension of the cohomology of $P_1$, as we pointed out in Section  3.5. In a recent work with the same co-authors we addressed the problem of the extension of the compatible infinitesimal deformations, obtaining similar negative results about the extendibility of many of the compatible brackets and a full classification of the compatible scalar Poisson brackets in two space dimensions \cite{CCS17}. Adopting a similar approach for the remaining two-component cases (namely $P_2$ and $P_{LP}$) is the next step in this research program.

\paragraph*{Acknowledgements}
The author was supported by the was supported by the INdAM-COFUND-2012 Marie Curie
fellowship ``MPoisCoho -- Poisson cohomology of multidimensional Hamiltonian operators"

He is grateful to the organizers of PMNP2017 for the opportunity to present his work and for the wonderful scientific environment they have established in Gallipoli. Most of the results presented in Section 3 were obtained in the final part of the author's PhD program and are part of his dissertation \cite{MC15-2}: this work couldn't have been done without the guidance and supervision of B.~Dubrovin, to whom I am most grateful. I wish also to thank R.~Vitolo for his help with REDUCE and \textsl{CDE} softwares, to D.~Robertz for providing me with an updated version of Janet, and the Department of Mathematics and Physics ``E.~De Giorgi'' of University of Salento for the access to their computer server.

\appendix
\section{Coboundary conditions for $H^2_3(P_{LP})$}
We provide here the Coboundary condition used in the computation of $H^2_3(P_{LP})$ as explained in Paragraph \ref{sec:H23PLP}.

{\footnotesize
\begin{multline}\label{eq:JComp1}
4p^3q^2A^{111}_{11} - 10p^4qA^{111}_{12} + 6p^5A^{111}_{22} + 6p^2q^3A^{112}_{11} - 18p^3q^2A^{112}_{12} + 12p^4qA^{112}_{22}\\
- 6p^2q^3A^{122}_{12} + 6p^3q^2A^{122}_{22} - 2q^5A^{222}_{11} + 2pq^4A^{222}_{12}=0
\end{multline}
\begin{multline}\label{eq:JComp2}
-2p^3q^2A^{111}_{11} + 6p^4qA^{111}_{12} - 4p^5A^{111}_{22} + 6p^3q^2A^{112}_{12} - 6p^4qA^{112}_{22} + 6pq^4A^{122}_{11}+\\
  - 6p^2q^3A^{122}_{12} + 4q^5A^{222}_{11} - 6pq^4A^{222}_{12} + 2p^2q^3A^{222}_{22}=0
 \end{multline}\begin{multline}\label{eq:JComp3}
14p^3q^2A^{111}_{11} - 80p^4qA^{111}_{12} + 80p^5A^{111}_{22} - 3p^2q^3A^{112}_{11} - 36p^3q^2A^{112}_{12} + 84p^4qA^{112}_{22}\\ 
 + 12pq^4A^{122}_{11} + 12p^2q^3A^{122}_{12} + 49q^5A^{222}_{11} - 8pq^4A^{222}_{12} + 16p^5qB^{11,11}_{12} + 12p^3q^3B^{11,22}_{12}\\
  - 8p^3q^3B^{12,21}_{12} + 8p^2q^4B^{12,22}_{12} - 16p^3q^3B^{22,11}_{12} + 16p^2q^4B^{22,12}_{12} + 8p^2q^4B^{22,21}_{12} + 28pq^5B^{22,22}_{12} \\
  + 8p^5qD^{1,111}_{12} - 16p^6D^{1,111}_{22} + 8p^3q^3D^{1,221}_{12} + 40pq^5D^{1,222}_{11} - 16p^3q^3D^{1,222}_{22} - 16p^3q^3D^{2,111}_{11} + 16p^4q^2D^{2,111}_{12}\\
   - 16p^5qD^{2,111}_{22} + 40q^6D^{2,222}_{11} - 32pq^5D^{2,222}_{12} - 17p^3q^3\dev_qA^{111}_{11} + 52p^4q^2\dev_qA^{111}_{12} - 12p^5q\dev_qA^{111}_{22}\\
    + 3p^2q^4\dev_qA^{112}_{11} + 66p^3q^3\dev_qA^{112}_{12} - 60p^4q^2\dev_qA^{112}_{22} + 57pq^5\dev_qA^{122}_{11} - 24p^2q^4\dev_qA^{122}_{12} - 24p^3q^3\dev_qA^{122}_{22}\\
     - 11q^6\dev_qA^{222}_{11} - 14pq^5\dev_qA^{222}_{12} - 2p^4q^2\dev_pA^{111}_{11} + 8p^6\dev_pA^{111}_{22} - 18p^3q^3\dev_pA^{112}_{11} + 24p^4q^2\dev_pA^{112}_{12}\\
      - 24p^5q\dev_pA^{112}_{22} - 30p^2q^4\dev_pA^{122}_{11} + 12p^3q^3\dev_pA^{122}_{12} - 70pq^5\dev_pA^{222}_{11} + 28p^2q^4\dev_pA^{222}_{12}=0
    \end{multline}
    \begin{multline}\label{eq:JComp4}
    -78p^3q^2A^{111}_{11} + 480p^4qA^{111}_{12} - 560p^5A^{111}_{22} - 9p^2q^3A^{112}_{11} + 228p^3q^2A^{112}_{12} - 516p^4qA^{112}_{22}\\
     - 180pq^4A^{122}_{11} - 12p^2q^3A^{122}_{12} - 413q^5A^{222}_{11} + 120pq^4A^{222}_{12} - 48p^5qB^{11,11}_{12} - 156p^3q^3B^{11,22}_{12}\\
      + 64p^4q^2B^{12,11}_{12} + 40p^3q^3B^{12,21}_{12} - 104p^2q^4B^{12,22}_{12} + 176p^3q^3B^{22,11}_{12} - 80p^2q^4B^{22,12}_{12} - 40p^2q^4B^{22,21}_{12}\\
       - 172pq^5B^{22,22}_{12} + 24p^5qD^{1,111}_{12} + 16p^6D^{1,111}_{22} - 40p^3q^3D^{1,221}_{12} - 264pq^5D^{1,222}_{11} + 144p^3q^3D^{1,222}_{22}\\
        + 144p^3q^3D^{2,111}_{11} - 80p^4q^2D^{2,111}_{12} + 16p^5qD^{2,111}_{22} - 264q^6D^{2,222}_{11} + 224pq^5D^{2,222}_{12} + 93p^3q^3\dev_qA^{111}_{11}\\
         - 420p^4q^2\dev_qA^{111}_{12} + 252p^5q\dev_qA^{111}_{22} - 39p^2q^4\dev_qA^{112}_{11} - 474p^3q^3\dev_qA^{112}_{12} + 492p^4q^2\dev_qA^{112}_{22} - 357pq^5\dev_qA^{122}_{11}\\
          + 120p^2q^4\dev_qA^{122}_{12} + 216p^3q^3\dev_qA^{122}_{22} + 79q^6\dev_qA^{222}_{11} + 86pq^5\dev_pA^{222}_{12} + 34p^4q^2\dev_pA^{111}_{11} - 48p^5q\dev_pA^{111}_{12}\\
          + 24p^6\dev_pA^{111}_{22} + 210p^3q^3\dev_pA^{112}_{11} - 216p^4q^2\dev_pA^{112}_{12} + 120p^5q\dev_pA^{112}_{22} + 318p^2q^4\dev_pA^{122}_{11} - 204p^3q^3\dev_pA^{122}_{12}\\
           + 518pq^5\dev_pA^{222}_{11}- 268p^2q^4\dev_pA^{222}_{12}=0
\end{multline}
}
\section{Basis of $H^2_3(P_{LP})$} \label{app:H23PLP}
{\footnotesize
\begin{subequations}
\begin{align}
\{p{}_{\mlambda}p\}^{(1)}&= 2\lambda_1^3  + \lambda_1\left(-\frac{p_yq_x}{q^2} + \frac{2pq_yq_x}{q^3} + \frac{3q_x^2}{q^2} - \frac{pq_{xy}}{q^2} - \frac{2q_{xx}}{q}\right)\\\notag
&\quad+\lambda_2\left(\frac{p_xq_x}{q^2} - \frac{pq_{xx}}{q^2}\right)+ \frac{p_yq_x^2}{q^3} - \frac{3pq_yq_x^2}{q^4} - \frac{3q_x^3}{q^3} + \frac{2pq_xq_{xy}}{q^3}\\ \notag
&\quad - \frac{p_yq_{xx}}{q^2} + \frac{2pq_yq_{xx}}{q^3} + \frac{4q_xq_{xx}}{q^2}  - \frac{pq_{xxy}}{q^2} - \frac{q_{xxx}}{q},\\
\{p{}_{\mlambda}q\}^{(1)}&=2\lambda_1^2\lambda_2+ \lambda_2\left(\frac{p_y^2}{2q^2} + \frac{p_yp_x}{pq} - \frac{2p_yq_x}{q^2} + \frac{5q_x^2}{2q^2} - \frac{2q_{xx}}{q}\right)\\\notag
&\quad + \lambda_1\left(\frac{3q_yq_x}{q^2} - \frac{2q_{xy}}{q}\right)- \frac{p_y^2q_y}{q^3} + \frac{p_yp_{yy}}{q^2} - \frac{p_y^2p_x}{p^2q} - \frac{p_yq_yp_x}{pq^2}\\\notag
&\quad + \frac{p_{yy}p_x}{pq} + \frac{4p_yq_yq_x}{q^3} - \frac{2p_{yy}q_x}{q^2} - \frac{6q_yq_x^2}{q^3} + \frac{p_yp_{xy}}{pq} - \frac{2p_yq_{xy}}{q^2}\\\notag
&\quad + \frac{6q_xq_{xy}}{q^2}  + \frac{2q_yq_{xx}}{q^2}  - \frac{2q_{xxy}}{q}\\
\{q{}_{\mlambda}p\}^{(1)}&=2\lambda_1^2\lambda_2  + \lambda_2\left(\frac{p_y^2}{2q^2} + \frac{p_yp_x}{pq} - \frac{2p_yq_x}{q^2}   + \frac{5q_x^2}{2q^2} - \frac{2q_{xx}}{q}\right)\\\notag
&\quad + \lambda_1\left(\frac{3q_yq_x}{q^2} - \frac{2q_{xy}}{q}\right)- \frac{5q_yq_x^2}{q^3}+ \frac{3q_yq_{xx}}{q^2} - \frac{2q_{xxy}}{q}+\frac{4q_xq_{xy}}{q^2}\\\notag
\{q{}_{\mlambda}q\}^{(1)}&=2\lambda_1\lambda_2^2+ \lambda_1\left(\frac{2p_yq_y}{pq} - \frac{p_{yy}}{p}\right)+ \lambda_2\left(\frac{2p_{yy}}{q}-\frac{2p_yq_y}{q^2} + \frac{p_yp_x}{p^2} + \frac{4q_yq_x}{q^2} - \frac{p_{xy}}{p} - \frac{4q_{xy}}{q}\right)\\
 &\quad+\frac{2p_yq_y^2}{q^3} - \frac{2q_yp_{yy}}{q^2} - \frac{p_yq_{yy}}{q^2} + \frac{p_{yyy}}{q} - \frac{p_y^2p_x}{p^3} - \frac{p_yq_yp_x}{p^2q} + \frac{p_{yy}p_x}{p^2}\\\notag
 &\quad - \frac{p_yq_yq_x}{pq^2} - \frac{4q_y^2q_x}{q^3} + \frac{2q_{yy}q_x}{q^2} + \frac{p_yp_{xy}}{p^2} + \frac{q_yp_{xy}}{pq} + \frac{p_yq_{xy}}{pq} + \frac{4q_yq_{xy}}{q^2}  - \frac{p_{xyy}}{p} -\frac{2q_{xyy}}{q},
\end{align}
\end{subequations}}
and
{\footnotesize
\begin{subequations}
\begin{align}
\{p{}_{\mlambda}p\}^{(2)}&= \lambda_1\left(\frac{53p_yq_y}{4q^2} - \frac{53p_{yy}}{4q} - \frac{35p_yp_x}{2p^2} - \frac{qp_x^2}{p^3} + \frac{q_yq_x}{2q^2} + \frac{19p_xq_x}{2p^2}\right.\\\notag
&\quad\qquad\left. + \frac{35p_{xy}}{2p} - \frac{q_{xy}}{2q} + \frac{qp_{xx}}{2p^2} - \frac{9q_{xx}}{p}\right)\\\notag
&\quad+ \lambda_2\left(\frac{17p_yp_x}{pq} -\frac{7p_y^2}{q^2}  + \frac{p_x^2}{2p^2} + \frac{59p_yq_x}{4q^2} - \frac{11p_xq_x}{pq} - \frac{3q_x^2}{2q^2} - \frac{71p_{xy}}{4q} + \frac{17q_{xx}}{2q}\right)\\\notag
&\quad+\frac{7p_y^2q_y}{q^3} - \frac{7p_yp_{yy}}{q^2} - \frac{17p_y^2p_x}{2p^2q} - \frac{17p_yq_yp_x}{2pq^2} + \frac{17p_{yy}p_x}{2pq} + \frac{17p_yp_x^2}{p^3} + \frac{3qp_x^3}{2p^4} - \frac{28p_yq_yq_x}{q^3}\\\notag
&\quad + \frac{14p_{yy}q_x}{q^2} + \frac{11p_yp_xq_x}{2p^2q} + \frac{11q_yp_xq_x}{2pq^2} - \frac{10p_x^2q_x}{p^3} + \frac{q_yq_x^2}{q^3} + \frac{17p_yp_{xy}}{2pq} + \frac{31q_yp_{xy}}{2q^2}\\\notag
&\quad - \frac{17p_xp_{xy}}{p^2} - \frac{11q_xp_{xy}}{2pq} + \frac{14p_yq_{xy}}{q^2} - \frac{11p_xq_{xy}}{2pq} - \frac{q_xq_{xy}}{q^2} - \frac{31p_{xyy}}{2q} - \frac{35p_yp_{xx}}{4p^2}\\\notag
&\quad - \frac{3qp_xp_{xx}}{2p^3} + \frac{5q_xp_{xx}}{p^2} - \frac{4q_yq_{xx}}{q^2} + \frac{37p_xq_{xx}}{4p^2}  + \frac{35p_{xxy}}{4p} + \frac{4q_{xxy}}{q} + \frac{qp_{xxx}}{4p^2} - \frac{9q_{xxx}}{2p},
\end{align}
\begin{align}
\{p{}_{\mlambda}q\}^{(2)}&=-\lambda_2^2\left(\frac{p_x}{p}-\frac{11p_y}{2q}\right) + \lambda_1\left(\frac{7p_yq_y}{pq} -\frac{17p_y^2}{2p^2} + \frac{9q_y^2}{4q^2} - \frac{9q_{yy}}{4q} + \frac{qp_yp_x}{p^3} - \frac{q_yp_x}{2p^2}\right.\\\notag
&\quad\qquad\left. + \frac{17p_yq_x}{2p^2} - \frac{19q_yq_x}{4pq} - \frac{qp_xq_x}{2p^3} - \frac{17q_x^2}{8p^2} - \frac{qp_{xy}}{p^2} + \frac{3q_{xy}}{2p} + \frac{qq_{xx}}{2p^2}\right)\\\notag
&\quad+ \lambda_2\left(\frac{3p_y^2}{2pq} + \frac{3p_yq_y}{2q^2} - \frac{3p_{yy}}{q} - \frac{p_yp_x}{2p^2} + \frac{3q_yp_x}{4pq} + \frac{3p_yq_x}{4pq} + \frac{3q_yq_x}{2q^2} + \frac{p_xq_x}{4p^2}\right.\\\notag
&\quad\qquad\left. - \frac{3q_x^2}{4pq} - \frac{3q_{xy}}{4q}\right)-\frac{3p_y^3}{2p^2q} - \frac{3p_y^2q_y}{2pq^2} + \frac{3p_yp_{yy}}{pq} + \frac{3q_yp_{yy}}{2q^2} - \frac{3p_{yyy}}{2q} + \frac{3p_yq_yp_x}{4p^2q}\\\notag
&\quad - \frac{3q_y^2p_x}{4pq^2} - \frac{p_{yy}p_x}{2p^2} + \frac{3q_{yy}p_x}{4pq} - \frac{3qp_yp_x^2}{p^4} + \frac{3q_yp_x^2}{2p^3}  - \frac{3p_y^2q_x}{4p^2q} + \frac{3p_yq_yq_x}{4pq^2} - \frac{9q_y^2q_x}{2q^3}\\\notag
&\quad + \frac{3p_{yy}q_x}{4pq} + \frac{9q_{yy}q_x}{4q^2} + \frac{5p_yp_xq_x}{2p^3} - \frac{3q_yp_xq_x}{4p^2q} + \frac{3qp_x^2q_x}{2p^4} + \frac{3p_yq_x^2}{4p^2q} - \frac{5p_xq_x^2}{4p^3} \\\notag
&\quad - \frac{3q_yp_{xy}}{4pq} + \frac{2qp_xp_{xy}}{p^3} - \frac{5q_xp_{xy}}{4p^2} - \frac{3p_yq_{xy}}{4pq} + \frac{15q_yq_{xy}}{4q^2} - \frac{5p_xq_{xy}}{4p^2} - \frac{3q_xq_{xy}}{4pq} \\\notag
&\quad + \frac{qp_yp_{xx}}{p^3} - \frac{3q_yp_{xx}}{4p^2} - \frac{qq_xp_{xx}}{2p^3} - \frac{p_yq_{xx}}{p^2} + \frac{3q_yq_{xx}}{4pq} - \frac{qp_xq_{xx}}{p^3} + \frac{q_xq_{xx}}{p^2}  - \frac{qp_{xxy}}{2p^2}\\\notag
&\quad  + \frac{p_yp_{xy}}{2p^2}- \frac{3q_{xyy}}{2q}+ \frac{3q_{xxy}}{4p} + \frac{qq_{xxx}}{4p^2},
\end{align}
\begin{align}
\{q{}_{\mlambda}p\}^{(2)}&=\lambda_2^2\left(\frac{p_x}{p}-\frac{11p_y}{2q}\right)+ \lambda_1\left(\frac{7p_yq_y}{pq} -\frac{17p_y^2}{2p^2} + \frac{9q_y^2}{4q^2} - \frac{9q_{yy}}{4q} + \frac{qp_yp_x}{p^3} - \frac{q_yp_x}{2p^2}\right.\\\notag
&\quad\qquad\left. + \frac{17p_yq_x}{2p^2} - \frac{19q_yq_x}{4pq} - \frac{qp_xq_x}{2p^3} - \frac{17q_x^2}{8p^2} - \frac{qp_{xy}}{p^2} + \frac{3q_{xy}}{2p} + \frac{qq_{xx}}{2p^2}\right)\\\notag
&\quad + \lambda_2\left(\frac{3p_y^2}{2pq} + \frac{25p_yq_y}{2q^2} - \frac{14p_{yy}}{q} - \frac{5p_yp_x}{2p^2} + \frac{3q_yp_x}{4pq} + \frac{3p_yq_x}{4pq} + \frac{3q_yq_x}{2q^2} + \frac{p_xq_x}{4p^2}\right.\\\notag
&\quad\qquad\left. - \frac{3q_x^2}{4pq} + \frac{2p_{xy}}{p} - \frac{3q_{xy}}{4q}\right)-\frac{14p_yq_y^2}{q^3} + \frac{14q_yp_{yy}}{q^2} + \frac{7p_yq_{yy}}{q^2} - \frac{7p_{yyy}}{q}\\\notag
&\quad + \frac{20p_y^2p_x}{p^3} - \frac{17p_yq_yp_x}{2p^2q} - \frac{p_{yy}p_x}{p^2} - \frac{q_yp_x^2}{2p^3}  - \frac{17p_yq_yq_x}{2pq^2} - \frac{3q_y^2q_x}{q^3} + \frac{3q_{yy}q_x}{2q^2} \\\notag
&\quad + \frac{11q_yp_xq_x}{2p^2q} + \frac{11q_yq_x^2}{2pq^2} + \frac{5p_xq_x^2}{p^3} - \frac{20p_yp_{xy}}{p^2} + \frac{17q_yp_{xy}}{2pq} + \frac{qp_xp_{xy}}{p^3} + \frac{9q_xp_{xy}}{p^2}  \\\notag
&\quad+ \frac{3q_yq_{xy}}{q^2} - \frac{p_xq_{xy}}{2p^2} - \frac{11q_xq_{xy}}{2pq}  + \frac{p_{xyy}}{p} - \frac{3q_{xyy}}{2q} + \frac{q_yp_{xx}}{4p^2} + \frac{19p_yq_{xx}}{2p^2} - \frac{11q_yq_{xx}}{2pq}\\\notag
&\quad - \frac{19p_yp_xq_x}{p^3}+ \frac{17p_yq_{xy}}{2pq}- \frac{qp_xq_{xx}}{2p^3} - \frac{19q_xq_{xx}}{4p^2}  - \frac{qp_{xxy}}{2p^2} + \frac{3q_{xxy}}{4p} + \frac{qq_{xxx}}{4p^2},
\end{align}
\begin{align}
\{q{}_{\mlambda}q\}^{(2)}&=\lambda_1\left(\frac{2qp_y^2}{p^3} - \frac{4p_yq_y}{p^2} + \frac{3q_y^2}{2pq} - \frac{2qp_yq_x}{p^3} + \frac{2q_yq_x}{p^2} + \frac{qq_x^2}{2p^3}\right)\\\notag
&\quad+ \lambda_2\left(\frac{3p_y^2}{p^2} + \frac{3p_yq_y}{pq} + \frac{3q_y^2}{q^2} - \frac{4p_{yy}}{p} - \frac{3q_{yy}}{q} - \frac{p_yq_x}{p^2} - \frac{3q_yq_x}{2pq} - \frac{q_x^2}{4p^2} + \frac{2q_{xy}}{p}\right)\\\notag
&\quad + \frac{5p_yp_{yy}}{p^2}-\frac{3p_y^3}{p^3} - \frac{3p_y^2q_y}{2p^2q} - \frac{3p_yq_y^2}{2pq^2} - \frac{3q_y^3}{q^3}  + \frac{3q_yp_{yy}}{2pq} + \frac{3p_yq_{yy}}{2pq} + \frac{9q_yq_{yy}}{2q^2}\\\notag
&\quad - \frac{2p_{yyy}}{p} - \frac{3p_{yyy}}{2q} - \frac{3qp_y^2p_x}{p^4} + \frac{4p_yq_yp_x}{p^3} - \frac{3q_y^2p_x}{4p^2q} + \frac{2p_y^2q_x}{p^3} + \frac{3p_yq_yq_x}{4p^2q} - \frac{p_{yy}q_x}{2p^2}\\\notag
&\quad - \frac{3q_{yy}q_x}{4pq} + \frac{3qp_yp_xq_x}{p^4} - \frac{2q_yp_xq_x}{p^3} - \frac{3p_yq_x^2}{4p^3} - \frac{3qp_xq_x^2}{4p^4} + \frac{q_x^3}{4p^3}  + \frac{2qp_yp_{xy}}{p^3} - \frac{2q_yp_{xy}}{p^2}\\\notag
&\quad - \frac{qq_xp_{xy}}{p^3} - \frac{7p_yq_{xy}}{2p^2} + \frac{3q_yq_{xy}}{4pq} + \frac{3q_xq_{xy}}{4p^2}  + \frac{q_{xyy}}{p} - \frac{qp_yq_{xx}}{p^3} + \frac{q_yq_{xx}}{p^2} + \frac{qq_xq_{xx}}{2p^3}.
\end{align}
\end{subequations}}

\bibliography{biblio}

\begin{thebibliography}{10}

\bibitem{BdSK09}
Aliaa Barakat, Alberto De~Sole, and Victor~G. Kac.
\newblock Poisson vertex algebras in the theory of {H}amiltonian equations.
\newblock {\em Jpn. J. Math.}, 4(2):141--252, 2009.

\bibitem{B73}
DJ~Benney.
\newblock Some properties of long nonlinear waves.
\newblock {\em Studies in Applied Mathematics}, 52(1):45--50, 1973.

\bibitem{CCS17}
Guido Carlet, Matteo Casati, and Sergey Shadrin.
\newblock Normal forms of dispersive scalar {P}oisson brackets with two
  independent variables.
\newblock arXiv/1707.03703, 2017.

\bibitem{CCS15}
Guido Carlet, Matteo Casati, and Sergey Shadrin.
\newblock Poisson cohomology of scalar multidimensional {D}ubrovin-{N}ovikov
  brackets.
\newblock {\em J. Geom. Phys}, 114(1):404--419, 2017.

\bibitem{MC15-2}
Matteo Casati.
\newblock {\em Multidimensional {P}oisson {V}ertex {A}lgebras and the {P}oisson
  cohomology of {H}amiltonian structures of hydrodynamic type}.
\newblock PhD thesis, Scuola Internazionale Superiore di Studi Avanzati di
  Trieste, 2015.

\bibitem{MC15-1}
Matteo Casati.
\newblock On deformations of multidimensional {P}oisson brackets of
  hydrodynamic type.
\newblock {\em Comm. Math. Phys.}, 335(2):851--894, 2015.

\bibitem{MC16}
Matteo Casati.
\newblock Dispersive deformations of the {H}amiltonian structure of {E}uler's
  equations.
\newblock {\em Theor. Math. Phys.}, (188):1296, 2016.

\bibitem{CV16}
Matteo Casati and Daniele Valeri.
\newblock Master{PVA} and {WA}lg: Mathematica packages for {P}oisson vertex
  algebras and classical affine {$\mathcal{W}$}-algebras.
\newblock {\em Bollettino dell'Unione Matematica Italiana}, 2017.

\bibitem{dSK13-2}
Alberto De~Sole and Victor~G. Kac.
\newblock The variational {P}oisson cohomology.
\newblock {\em Jpn. J. Math.}, 8(1):1--145, 2013.

\bibitem{dGMS05}
Luca Degiovanni, Franco Magri, and Vincenzo Sciacca.
\newblock On deformation of {P}oisson manifolds of hydrodynamic type.
\newblock {\em Comm. Math. Phys.}, 253(1):1--24, 2005.

\bibitem{DN83}
B.~A. Dubrovin and S.~P. Novikov.
\newblock Hamiltonian formalism of one-dimensional systems of the hydrodynamic
  type and the {B}ogolyubov-{W}hitham averaging method.
\newblock {\em Dokl. Akad. Nauk SSSR}, 270(4):781--785, 1983.

\bibitem{DN84}
B.~A. Dubrovin and S.~P. Novikov.
\newblock Poisson brackets of hydrodynamic type.
\newblock {\em Dokl. Akad. Nauk SSSR}, 279(2):294--297, 1984.

\bibitem{DZ}
Boris~A. Dubrovin and Youjin Zhang.
\newblock {N}ormal forms of hierarchies of integrable {PDE}s, {F}robenius
  manifolds and {G}romov-{W}itten invariants.
\newblock arXiv:math/0108160v1, 2001.

\bibitem{FOS11}
E.~V. Ferapontov, A.~V. Odesskii, and N.~M. Stoilov.
\newblock Classification of integrable two-component {H}amiltonian systems of
  hydrodynamic type in {$2+1$} dimensions.
\newblock {\em J. Math. Phys.}, 52(7):073505, 28, 2011.

\bibitem{FLS15}
Evgeny~V. Ferapontov, Paolo Lorenzoni, and Andrea Savoldi.
\newblock Hamiltonian operators of {D}ubrovin-{N}ovikov type in 2{D}.
\newblock {\em Lett. Math. Phys.}, 105(3):341--377, 2015.

\bibitem{g71}
Clifford~S. Gardner.
\newblock Korteweg-de {V}ries equation and generalizations. {IV}. {T}he
  {K}orteweg-de {V}ries equation as a {H}amiltonian system.
\newblock {\em J. Mathematical Phys.}, 12:1548--1551, 1971.

\bibitem{GD75}
I.~M. Gel{$'$}fand and L.~A. Diki{\u\i}.
\newblock Asymptotic properties of the resolvent of {S}turm-{L}iouville
  equations, and the algebra of {K}orteweg-de {V}ries equations.
\newblock {\em Uspehi Mat. Nauk}, 30(5(185)):67--100, 1975.

\bibitem{G01}
Ezra Getzler.
\newblock A {D}arboux theorem for {H}amiltonian operators in the formal
  calculus of variations.
\newblock {\em Duke Math. J.}, 111(3):535--560, 2002.

\bibitem{L77}
Andr{\'e} Lichnerowicz.
\newblock Les vari\'et\'es de {P}oisson et leurs alg\`ebres de {L}ie
  associ\'ees.
\newblock {\em J. Differential Geometry}, 12(2):253--300, 1977.

\bibitem{M80}
F.~Magri.
\newblock A geometrical approach to the nonlinear solvable equations.
\newblock In {\em Nonlinear evolution equations and dynamical systems ({P}roc.
  {M}eeting, {U}niv. {L}ecce, {L}ecce, 1979)}, volume 120 of {\em Lecture Notes
  in Phys.}, pages 233--263. Springer, Berlin-New York, 1980.

\bibitem{M78}
Franco Magri.
\newblock A simple model of the integrable {H}amiltonian equation.
\newblock {\em J. Math. Phys.}, 19(5):1156--1162, 1978.

\bibitem{M88}
O.~I. Mokhov.
\newblock Poisson brackets of {D}ubrovin-{N}ovikov type ({DN}-brackets).
\newblock {\em Funktsional. Anal. i Prilozhen.}, 22(4):92--93, 1988.

\bibitem{M08}
O.~I. Mokhov.
\newblock Classification of nonsingular multidimensional {D}ubrovin-{N}ovikov
  brackets.
\newblock {\em Funktsional. Anal. i Prilozhen.}, 42(1):39--52, 95--96, 2008.

\bibitem{N82}
Sergej~P. Novikov.
\newblock The {H}amiltonian formalism and a multivalued analogue of {M}orse
  theory.
\newblock {\em Uspekhi Mat. Nauk}, 37(5(227)):3--49, 248, 1982.

\bibitem{PR05}
W.~Plesken and D.~Robertz.
\newblock Janet's approach to presentations and resolutions for polynomials and
  linear {PDE}s.
\newblock {\em Arch. Math. (Basel)}, 84(1):22--37, 2005.

\bibitem{V14}
Raffaele Vitolo.
\newblock {CDE}: A {REDUCE} package for integrability of {PDE}s.
\newblock \url{http://gdeq.org/files/Cde-userguide-1.0.pdf}, 2014.

\bibitem{V10}
Raffaele Vitolo, Paul H.~M. Kersten, Gerhard Post, and G.~Roelofs.
\newblock {CDIFF}: A {REDUCE} package for computations in geometry of
  differential equations.
\newblock \url{http://gdeq.org/files/Cdiff-userguide-3.pdf}, 2010.

\end{thebibliography}

\end{document}